\documentclass[reqno]{amsart}
\usepackage[
left=1.25in, right=1.25in]{geometry}
\usepackage{tikz}
\usepackage[all]{xy}
\usepackage{graphicx}
\usepackage{mathrsfs}
\usepackage{hyperref}
\usepackage{microtype}
\usetikzlibrary{arrows}

\newcommand{\ba}{\begin{align}}
\newcommand{\ea}{\end{align}}
\newcommand{\be}{\begin{equation}}
\newcommand{\ee}{\end{equation}}

\newcommand{\beq}{\begin{eqnarray}}
\newcommand{\eeq}{\end{eqnarray}}

\newcommand{\beqs}{\begin{eqnarray*}}
\newcommand{\eeqs}{\end{eqnarray*}}

\newcommand\SA{\mathscr{S}}
\newcommand\F{\mathfrak{F}}
\newcommand\D{\mathfrak{D}}

\newcommand\C{\mathscr{C}}
\newcommand\Z{\mathbb{Z}}
\newcommand\A{\mathfrak{A}}

\begin{document}

\theoremstyle{plain}
\newtheorem{theorem}{Theorem~}[section]
\newtheorem{main}{Main Theorem~}
\newtheorem{lemma}[theorem]{Lemma~}
\newtheorem{proposition}[theorem]{Proposition~}
\newtheorem{corollary}[theorem]{Corollary~}
\newtheorem{definition}[theorem]{Definition~}
\newtheorem{notation}[theorem]{Notation~}
\newtheorem{example}[theorem]{Example~}
\newtheorem{remark}[theorem]{Remark~}
\newtheorem*{question}{Question}
\newtheorem*{claim}{Claim}
\newtheorem*{ac}{Acknowledgement}
\newtheorem*{conjecture}{Conjecture~}
\renewcommand{\proofname}{\bf Proof}



\newcommand{\rotateRPY}[3]
{   \pgfmathsetmacro{\rollangle}{#1}
    \pgfmathsetmacro{\pitchangle}{#2}
    \pgfmathsetmacro{\yawangle}{#3}

    \pgfmathsetmacro{\newxx}{cos(\yawangle)*cos(\pitchangle)}
    \pgfmathsetmacro{\newxy}{sin(\yawangle)*cos(\pitchangle)}
    \pgfmathsetmacro{\newxz}{-sin(\pitchangle)}
    \path (\newxx,\newxy,\newxz);
    \pgfgetlastxy{\nxx}{\nxy};

    \pgfmathsetmacro{\newyx}{cos(\yawangle)*sin(\pitchangle)*sin(\rollangle)-sin(\yawangle)*cos(\rollangle)}
    \pgfmathsetmacro{\newyy}{sin(\yawangle)*sin(\pitchangle)*sin(\rollangle)+ cos(\yawangle)*cos(\rollangle)}
    \pgfmathsetmacro{\newyz}{cos(\pitchangle)*sin(\rollangle)}
    \path (\newyx,\newyy,\newyz);
    \pgfgetlastxy{\nyx}{\nyy};

    \pgfmathsetmacro{\newzx}{cos(\yawangle)*sin(\pitchangle)*cos(\rollangle)+ sin(\yawangle)*sin(\rollangle)}
    \pgfmathsetmacro{\newzy}{sin(\yawangle)*sin(\pitchangle)*cos(\rollangle)-cos(\yawangle)*sin(\rollangle)}
    \pgfmathsetmacro{\newzz}{cos(\pitchangle)*cos(\rollangle)}
    \path (\newzx,\newzy,\newzz);
    \pgfgetlastxy{\nzx}{\nzy};
}

\newcommand{\myGT}[2]
{
    \pgftransformcm{2}{0}{0}{3}{\pgfpoint{#1cm}{#2cm}}
}



\tikzstyle WL=[line width=3pt,opacity=1.0]

\newcommand{\drawWL}[3]
{
    \draw[white,WL]  (#2) -- (#3);
    \draw[#1] (#2) -- (#3);
}



\newpage
\title{
Jones-Wassermann subfactors for modular tensor categories
}
\author{Zhengwei Liu}
\address{Department of Mathematics and Department of Physics\\Harvard University}
\email{zhengweiliu@fas.harvard.edu}
\author{Feng Xu}
\address{Department of Mathematics\\
 University of California, Riverside}
\email{xufeng@math.ucr.edu}


\begin{abstract}

The representation theory of a conformal net is a unitary modular tensor category. It is captured by the bimodule category of the Jones-Wassermann subfactor.
In this paper, we construct multi-interval Jones-Wassermann subfactors for unitary modular tensor categories.
We prove that these subfactors are self-dual. It generalizes and categorifies the self-duality of finite abelian groups and we call it modular self-duality.
\end{abstract}

\maketitle

\section{Introduction}

Subfactor theory provides an entry point into a world of mathematics
and physics containing large parts of conformal field theory,
quantum algebras and low dimensional topology (cf. \cite{JonICM}) and
references therein).  In \cite{Jon14} V. Jones has devised a
renormalization program based on planar algebras as an attempt to
show that all finite depth subfactors are related to CFT, i.e.,  the
double of a finite depth subfactor should be related to CFT. 

More generally, the program is the following: given a unitary modular
tensor category (MTC) $\C$, (cf. \cite{Turaev}), can we construct a CFT
whose representation category is isomorphic to $\C$? We shall call
such a program ``reconstruction program", analogue to a similar
program in higher dimensions by Doplicher-Roberts (cf. \cite{DopRob89}).

Given a rational conformal net $\A$, and let $I$ be a union of $n>1
$ disconnected intervals. The Jones-Wassermann subfactor is the
subfactor $\A(I)\subset \A(I')'$ \cite{Lon95,Was98,Xu00,KawLonMug01}. This subfactor is related to
permutation orbifold and a simple application of orbifold theory
shows that the Jones-Wassermann subfactor is self-dual, see Remark \ref{Rem:orbifold} and \cite{KacLonXu05}. 

If the reconstruction program works, then for any MTC $\C$ we can find a
rational conformal net $\A$ such that the category of
representations of $\A$ is isomorphic to $\C$, it will follow that
there are self-dual Jones-Wassermann subfactors for each integer
$m>1$. Hence a positive solution to reconstruction program would
imply that we can construct self-dual Jones-Wassermann subfactors
for each integer $m>1$ associated with any unitary MTC $\C$.  This
is the motivation for our paper. 

Our main result gives a construction of  self-dual Jones-Wassermann subfactors for each
integer $m>1$ associated with any unitary MTC $\C$. The main difficulty is the proof of the self-duality for Jones-Wassermann subfactors for MTC.
The proof of the self-duality essentially requires the modularity of $\C$, so we call it the \emph{modular self-duality}.
We believe that our construction will shed new light on the reconstruction program.

We construct the ``$m$-interval'' Jones-Wassermann subfactor associated with a unitary MTC $\C$ by a Frobenius algebra $\gamma_m$ in $\C^m$, the $m^{\rm th}$ tensor power of $\C$, although there is no notion of intervals in modular tensor categories. We give an explicit formula for the objects and morphisms of these Frobenius algebras. When $m=2$, the Jones-Wassermann subfactor defines the quantum double of $\C$ \cite{Dri86,Ocn91,Pop94,Lon95,Mug03II}.

The bimodule category of a subfactor is described by a subfactor planar algebra \cite{JonPA}.
The $n$-box space of the planar algebra of the $m$-interval Jones-Wassermann subfactor for $\C$ is given by the vector space $\hom_{\C^m}(1,\gamma_m^n)$.
It turns out to be natural to represent these vectors by a 3D picture.
This representation identifies $\hom_{\C^m}(1,\gamma_m^n)$ as a configuration space $Conf_{n,m}$ on a 2D $n\times m$ lattice.
Therefore the configuration space $\{Conf_{n,m}\}_{m,n\in \mathbb{N}}$ unifies the Jones-Wassermann subfactors for all $m\geq 1$. It is a natural candidate for the configuration space of a 2D lattice model that can be used in the reconstruction program.

Moreover, we show that planar tangles can act on $\{Conf_{n,m}\}_{m,n\in \mathbb{N}}$ in two different directions independently. 
In one direction $m$ is fixed. These actions are the usual ones in the planar algebra of the $m$-interval Jones-Wassermann subfactor. In the other direction $n$ is fixed. These actions relate the Jones-Wassermann subfactor with different intervals which have not been studied before.

The bi-directional actions of planar tangles are compatible with the geometric actions on the 2D lattices. We call such family of vector spaces a \emph{bi-planar algebra}. It is a new subject in subfactor theory and it adds one additional dimension to the theory of planar algebras.

This 3D representation also leads to the discovery of a new symmetry between $m$ and $n$, although the meaning of the actions of planar tangles in the two directions are completely different.  
It will be interesting to understand these additional symmetries in conformal field theory.

When $\C$ is the representation category of a finite abelian group $G$, the configuration space $Conf(\C)_{2,2}$ becomes $L^2(G)$. Moreover, the modular self-duality coincides with the self-duality of $G$. The proof of the self-duality of $G$ uses the discrete Fourier transform on $G$. 
We construct the string Fourier transform (SFT) on the configuration space $\hom_{\C^m}(1,\gamma_m^n)$ to prove the modular self-duality. From this point of view, the modular self-duality and the SFT generalize and categorify the self-duality and the Fourier transform of finite abelian groups.

Moreover, the SFT on $Conf(\C)_{2,2}$ is the same as the modular $S$-matrix of $\C$.
Therefore we can study the Fourier analysis of the $S$-matrix on $Conf(\C)_{2,2}$. It fits into the recent progress about the Fourier analysis on subfactors \cite{Liuex,JLW16,LW-survey,JLW16b}.

The modular self-duality has been used in the quon 3D language for quantum information \cite{JLW-Quon}, where the vector space $\hom_{\C^2}(1,\gamma_2^2)$ is considered as the 1-quon space. A combination of these two works leads to further applications in the study of MTC.


\begin{ac}
Zhengwei Liu was supported by a grant from Templeton Religion Trust and an AMS-Simons Travel Grant.
Feng Xu was supported in part by an academic senate grant from UCR.
The authors would like to thank Vaughan F. R. Jones for stimulating discussions abour his renormalization program.
\end{ac}

\section{Configuration spaces}
\subsection{Modular tensor categories}
We refer the readers to \cite{Turaev} for basic definitions about modular tensor categories.
Suppose $\C$ is a unitary modular tensor category.
Let $Irr$ be the set of simple objects of $\C$ and the unit is denoted by 1.
For an object $X$, its dual object is denoted by $\overline{X}$. Its quantum dimension is $d(X)$.
Let $\displaystyle \mu=\sum_{X\in Irr} d(X)^2$ be the global dimension of $\C$.

The modular conjugation $\theta_{\C}$ on $\C$ is a horizontal reflection. We have that $\theta_{\C}(X)=\overline{X}$. Moreover, for objects $X$, $Y$, $Z$ in $\C$, $\theta_{\C}:\hom(X\otimes Y,Z) \to \hom(\overline{Y}\otimes \overline{X}, \overline{Z})$ is an anti-linear algebroid isomorphism.
The adjoint operator $*$ on $\C$ is a vertical reflection. We have that $X^*=X$. Moreover, $*:\hom(X\otimes Y, Z)\to \hom(Z, X\otimes Y)$ is an anti-linear algebroid anti-isomorphism.
The contragredient map $\rho_{\pi}$ on $\C$ is a rotation by $\pi$. We have that  $\rho(X)=\overline{X}$. Moreover, $\rho: \hom(X\otimes Y, Z) \to \hom(\overline{Z}, \overline{Y}\otimes \overline{X})$ is a linear algebroid anti-isomorphism.
Furthermore
\begin{align}\label{Equ:theta=rho*}
\theta_{\C}=\rho_{\pi}\circ *.
\end{align}

We can identify the morphism spaces $\hom(\overline{Z},X\otimes Y)$ and $ \hom(1,X\otimes Y\otimes Z)$ as follows:
For a morphism $\alpha\in \hom(1,X\otimes Y\otimes Z)$, we obtain a morphism $\tilde{\alpha}=(1_X\otimes 1_Y \otimes \phi_{Z\otimes \overline{Z}}) (\alpha\otimes 1_{\overline{Z}} )$ in $\hom(\overline{Z}, X\otimes Y)$, where $\phi_{Z\otimes \overline{Z}} \in \hom(1, Z\otimes \overline{Z})$ is the duality map.
\begin{notation}[Frobenius reciprocity]
Diagrammatically we represent $\tilde{\alpha}$ as

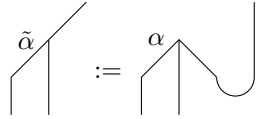
\begin{align*}
\raisebox{-.5cm}{
\begin{tikzpicture}
\node at (.2,.5) {$\tilde{\alpha}$};
  \draw (0.5,.5)--(0,0)--++(0,-.5);
  \draw (0.5,.5)--(0.5,0)--++(0,-.5);
  \draw (0.5,.5)--(1,1);
\end{tikzpicture}}
&:=
\raisebox{-.5cm}{
\begin{tikzpicture}
\node at (.2,.5) {$\alpha$};
  \draw (0.5,.5)--(0,0)--++(0,-.5);
  \draw (0.5,.5)--(0.5,0)--++(0,-.5);
  \draw (0.5,.5)--(1,0);
\draw (1,0) arc (-180:0:.25) --++(0,1);
\end{tikzpicture}}
\end{align*}
\end{notation}

\begin{notation}
For an object $X$ in $\C$, we denote an ortho-normal-basis of $\hom_{\C}(1,X)$ by $ONB_{\C}(X)$, or $ONB(X)$ for short.
We denote an ortho-normal-basis of $\hom_{\C}(X,1)$ by $ONB^*_{\C}(X)$, or $ONB^*(X)$ for short.
\end{notation}

For two objects  $X$ and $Y$, we have the resolution of the identity:
\be \label{Equ:resolution of identity}
1_X\otimes 1_Y=
\raisebox{-.5cm}{
\begin{tikzpicture}
  \draw (0,0)--(0,1);
  \draw (0.5,0)--(0.5,1);
\end{tikzpicture}
}
=\sum_{Z\in Irr,\alpha\in ONB(X\otimes Y \otimes Z)} d(Z)
\raisebox{-.5cm}{
\begin{tikzpicture}
\node at (-.1,.75) {$\alpha^*$};
\node at (-.1,.25) {$\alpha$};
  \draw (0,0)--(.25,.25)--(.25,.75)--(0,1);
  \draw (0.5,0)--(.25,.25)--(.25,.75)--(0.5,1);
\end{tikzpicture}
}
\ee

\subsection{Configuration spaces}
Now let us define the configuration space on a finite 2D-lattice with the target space $\C$.
Each configuration has three parts: $Z$-, $X$-, $Y$- configurations.

We use $Grid(n,m)$ to represent the grid $\Z_n\times \Z_m \times \{\pm1\}$.
We allocate the vertices of the grid at $(i,j,\pm 1)$, $0\leq i\leq m-1$ and $1\leq j\leq n$ in the 3D space which are indicated by the bullets in Fig.~\ref{Fig:grid}.
To simplify the notations, we draw pictures for $n=4$, $m=3$. The reader can figure out the general case.


\begin{figure}[h]
\begin{tikzpicture}
\myGT{3}{1};
\draw[->] (-1,0,0)--(4,0,0) node[above] {$X$};
\draw[->] (0,-.4,0)--(0,.4,0) node[above] {$Z$};
\draw[->] (0,0,1)--(0,0,-3) node[above] {$Y$};
\foreach \x in {1,2,3,4}{
\foreach \z in {0,1,2}{
\foreach \y in {-.15,.15}{
\node[blue] at (\x,\y,-\z) {$\bullet$};
}}}
\end{tikzpicture}
\caption{Grid(n,m) for $n=4$, $m=3$.}  \label{Fig:grid}
\end{figure}
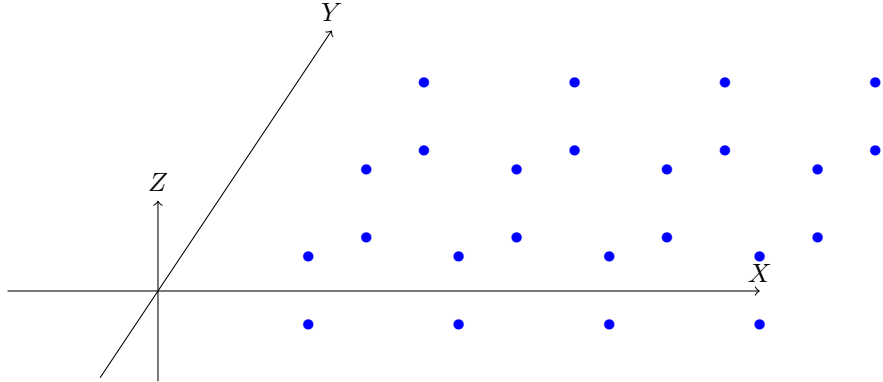

For the lattice $Lat=\Z_n\times \Z_m$,
a $Z$-configuration is a map from the sites of the lattice to simple objects in $\C$. We denote the simple object at the site $(i,j)$ as $X_{i,j}$.
We denote this $Z$-configuration by $X_{\vec{i},\vec{j}}$ and represent it in the 3D space by assigning the object $X_{i,j}$ to the line from $(i,j,1)$ to $(i,j,-1)$ as in Fig.~\ref{Fig:Z-configuration}:
\begin{figure}[h]\raisebox{0cm}{
\begin{tikzpicture}
\myGT{3}{1};
            \foreach \x in {1,2,3,4} {
               \foreach \z in {0,1,2} {
                    \drawWL{blue}{\x,0,-\z}{\x,-0.3,-\z};
                 }
                }
   \foreach \z in {0,1,2} {
    \foreach \x in {1,2,3,4} {
    \node[blue] at (\x,0,-\z) {$\bullet$};
    \node[blue] at (\x,-0.3,-\z) {$\bullet$};
    \node at (\x-.2,-0.15,-\z) {$X_{\x,\z}$};
\draw[blue] (\x-.05,-0.15+.05,-\z)--(\x,-0.15,-\z)--(\x+.05,-0.15+.05,-\z);
    }}
\end{tikzpicture}
}.\caption{Z-Configuration.}  \label{Fig:Z-configuration}
\end{figure}
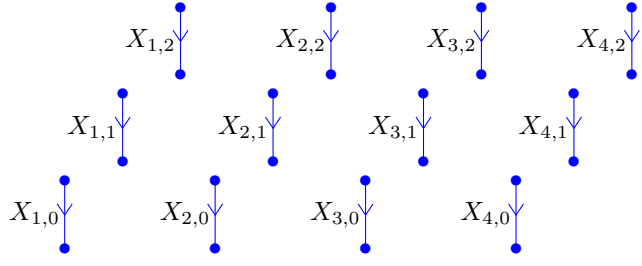

We denote $X_{\vec{i},j}=X_{1,j}\otimes \cdots \otimes X_{n,j}$ and $X_{i,\vec{j}}=X_{i,0}\otimes \cdots \otimes X_{i,m-1}$.
Moreover,
\begin{align*}
d(X_{\vec{i},\vec{j}})&:=\prod_{1\leq i \leq n,0\leq j\leq m-1} d(X_{i,j}), \\
d(X_{\vec{i},j})&:=\prod_{1\leq i \leq n} d(X_{i,j}), \\
d(X_{i,\vec{j}})&:=\prod_{0\leq j\leq m-1} d(X_{i,j}).
\end{align*}

An $X$-configuration with boundary $X_{\vec{i},j}$ is a morphism  $a_j$  in $\hom(1, X_{\vec{i},j})$. We denote the boundary by $X(a_j):=X_{\vec{i},j}$.
A $Y$-configuration with boundary $X_{i,\vec{j}}$ is a morphism $b_i$ in $\hom(X_{i,\vec{j}},1)$. We denote the boundary by $X(b_i):=X_{i,\vec{j}}$.
We represent them in the 3D space in Fig.~\ref{Fig:XY-configurations}.
Moreover, we call $a_{\vec{j}}=a_0\otimes \cdots \otimes a_m$ an $X$-configuration with boundary $X_{\vec{i},\vec{j}}$ and
 $b_{\vec{i}}=b_1\otimes\cdots \otimes b_n$ a $Y$-configuration with boundary $X_{\vec{i},\vec{j}}$.

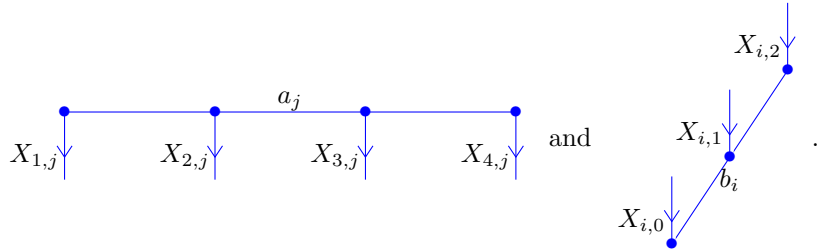
\begin{figure}[h]
\raisebox{-.5cm}{
\begin{tikzpicture}
\myGT{3}{1};
\node at (2.5,.05,0) {$a_j$};
\draw [blue] (1,0,0)--(4,0,0);
\foreach \x in {1,2,3,4} {
{
\drawWL{blue} {\x,0,0}{\x,-0.3,0};
\draw[blue] (\x-.05,-0.2+.05,0)--(\x,-0.2,0)--(\x+.05,-0.2+.05,0);
\node at (\x-.2,-0.2,0) {$X_{\x,j}$};
\node[blue] at (\x,0,0) {$\bullet$};
}}
\end{tikzpicture}
}
and
\raisebox{-1.5cm}{
\begin{tikzpicture}
\myGT{3}{1};
\draw [blue] (0,-0.6,0)--(0,-0.6,-2);
\node at (0,-.7,-1) {$b_i$};
\foreach \z in {0,1,2}
{
\drawWL{blue} {0,-.3,-\z}{0,-0.6,-\z};
\draw[blue] (-.05,-0.5+.05,-\z)--(0,-0.5,-\z)--(.05,-0.5+.05,-\z);
\node at (-.2,-0.5,-\z) {$X_{i,\z}$};
\node[blue] at (0,-0.6,-\z) {$\bullet$};
}
\end{tikzpicture}
}.
\caption{$X$- and $Y$-Configurations.}  \label{Fig:XY-configurations}
\end{figure}

We call $a_{\vec{j}}\otimes b_{\vec{i}}$ a configuration with boundary $X_{\vec{i},\vec{j}}$, denoted by $X(a_{\vec{j}}\otimes b_{\vec{i}}):=X_{\vec{i},\vec{j}}$.
We represent it in the 3D space as in Fig.~\ref{Fig:configuration}:
We define the configuration space on the $n\times m$ 2D-lattice $Lat$ to be the Hilbert space
\begin{align*}
Conf(Lat)=Conf(\C)_{m,n}:&=\bigoplus_{X_{\vec{i},\vec{j}}\in Irr^{nm}} \left(\bigotimes_{j=0}^{m-1} \hom(1,X_{\vec{i},j}) \otimes \bigotimes_{i=1}^{n}  \hom(X_{i,\vec{j}},1) \right),
\end{align*}
where each hom space is considered as a Hilbert space.
We simply use the notation $\sum a_{\vec{j}}\otimes b_{\vec{i}}$ to represent an element in $Conf(Lat)$.

\begin{figure}[h]
\raisebox{0cm}{
\begin{tikzpicture}
\myGT{3}{1};
\node at (-1,0,-1) {$\displaystyle \textcolor{black}{\mu^{\frac{(1-n)(m-1)}{4}}} \sqrt{d(X_{\vec{i},\vec{j}})}$};
            \foreach \x in {1,2,3,4} {
            \draw[blue](\x,-0.3,0)--(\x,-0.3,-2);
            \node at (\x,-.5,-1) {$b_{\x}$};
               \foreach \z in {0,1,2} {
                 }
                }
               \foreach \z in {0,1,2} {
                \drawWL{blue} {1,0,-\z}{4,0,-\z};
                \node at (2.5,.05,-\z) {$a_{\z}$};
                 }
   \foreach \z in {0,1,2} {
    \foreach \x in {1,2,3,4} {
    \node[blue] at (\x,0,-\z) {$\bullet$};
    \node[blue] at (\x,-0.3,-\z) {$\bullet$};
                    \draw[blue] (\x,-.15,-\z)--(\x,-0.3,-\z);
                    \fill[green!20] (\x-.1,-.15,-\z+.1)--++(.2,0,0)--++(0,0,-.2)--++(-.2,0,0)--++(0,0,.2);
                    \draw[blue] (\x,0,-\z)--(\x,-0.15,-\z);
    \node at (\x-.2,-0.15,-\z) {$X_{\x,\z}$};
\draw[blue] (\x-.05,-0.15+.05,-\z)--(\x,-0.15,-\z)--(\x+.05,-0.15+.05,-\z);
    }}
\end{tikzpicture}
}.
\caption{Configurations for $n=4$, $m=3$: We use the small square at (i,j,0) to indicate the vertex (i,j) in the lattice $\Z_m\times \Z_n$.
}  \label{Fig:configuration}
\end{figure}
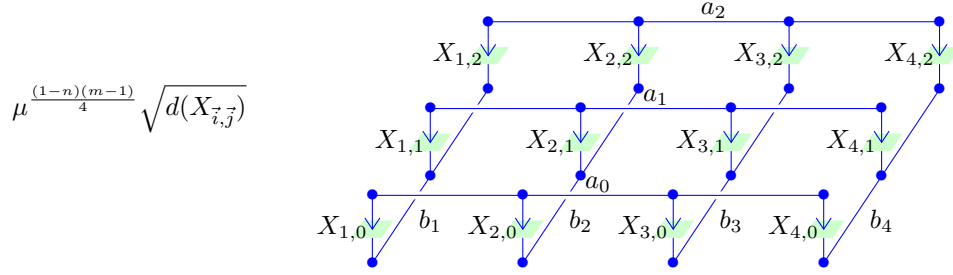

\subsection{Duality}\label{Sec:Duality}

When we consider $Lat=\Z_n \times \Z_m$ as a lattice on a torus, its dual lattice $Lat'$ is also $\Z_n\times \Z_m$ and the configuration space on the dual lattice is $Conf(\C)_{m,n}$.
We allocate the vertices of the corresponding $Grid(n,m)$ at $(i+1/2,j-1/2,\pm1)$, $0\leq i\leq m-1$ and $1\leq j\leq n$ in the 3D space.

We define a bilinear form $LL$ on the configuration spaces of the lattice and the dual lattice $Conf(Lat)\otimes Conf(Lat')=Conf(\C)_{m,n}\otimes Conf(\C)_{m,n}$.
For $a_{\vec{j}}\otimes b_{\vec{i}}$ with boundary $X_{\vec{i},\vec{j}}$ in $Conf(Lat)$, and $a'_{\vec{j}}\otimes b'_{\vec{i}}$ with boundary $X'_{\vec{i},\vec{j}}$ in $Conf(Lat')$, the bilinear form $LL$ is defined as

\begin{align} \nonumber
LL(a_{\vec{j}}\otimes b_{\vec{i}}, a'_{\vec{j}}\otimes b'_{\vec{i}}) =& \textcolor{black}{\mu^{\frac{(1-n)(m-1)}{2}}} \sqrt{d(X_{\vec{i},\vec{j}})d(X'_{\vec{i},\vec{j}})}
\\\label{Equ:bilinear form}
&\raisebox{0cm}
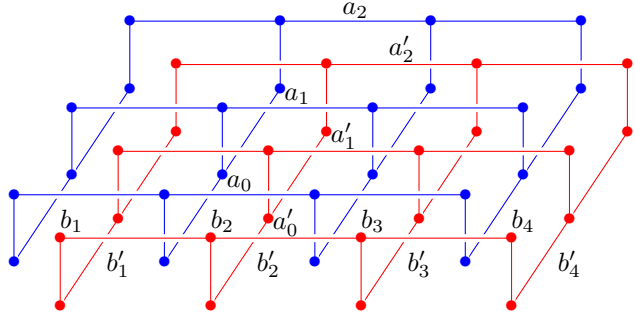
{
\begin{tikzpicture}
\myGT{3}{1};
            \foreach \x in {1,2,3,4} {
\begin{scope}[shift={(.5,0,.5)}]
            \draw[red](\x,-0.3,0)--(\x,-0.3,-2);
            \node at (\x,-.5,-1) {$b'_{\x}$};
\end{scope}
            \draw[blue](\x,-0.3,0)--(\x,-0.3,-2);
            \node at (\x,-.5,-1) {$b_{\x}$};
               \foreach \z in {0,1,2} {
\begin{scope}[shift={(.5,0,.5)}]
                    \drawWL{red}{\x,0,-\z}{\x,-0.3,-\z};
\end{scope}
                    \drawWL{blue}{\x,0,-\z}{\x,-0.3,-\z};
                 }
                }
               \foreach \z in {0,1,2} {
\begin{scope}[shift={(.5,0,.5)}]
                \drawWL{red} {1,0,-\z}{4,0,-\z};
                \node at (2.5,.07,-\z) {$a'_{\z}$};
\end{scope}
                \drawWL{blue} {1,0,-\z}{4,0,-\z};
                \node at (2.5,.05,-\z) {$a_{\z}$};
                 }
   \foreach \z in {0,1,2} {
    \foreach \x in {1,2,3,4} {
\begin{scope}[shift={(.5,0,.5)}]
    \node[red] at (\x,0,-\z) {$\bullet$};
    \node[red] at (\x,-0.3,-\z) {$\bullet$};
\end{scope}
    \node[blue] at (\x,0,-\z) {$\bullet$};
    \node[blue] at (\x,-0.3,-\z) {$\bullet$};
    }}
\end{tikzpicture}
}
\end{align}

When $m=0$ or $n=0$, we define the configuration space as the ground field. We define $LL$ as the multiplication of the two scalars.

\begin{theorem}\label{Thm:MN Self-duality}
The configurations spaces of the lattice and the dual lattice are dual to each other.
Precisely the map from $Conf(Lat)$ to the dual space of $Conf(Lat')$ induced by $LL(-,-)$ is an isometry.
\end{theorem}

We first prove the case for $m=n=2$. We prove the general case by a bi-induction in the rest of the paper. The order of the proofs is shown at the end of this Section.

\begin{proof}[Proof for the case $m=n=2$:]
 When $m=n=2$, the diagram in Equation \eqref{Equ:bilinear form} becomes the Hopf link and $LL$ defines the $S$ matrix of $\C$.

\begin{center}
\raisebox{0cm}{
\begin{tikzpicture}
\myGT{3}{1};
            \foreach \x in {1,2} {
\begin{scope}[shift={(.5,0,.5)}]
            \draw[red](\x,-0.3,0)--(\x,-0.3,-1);
\end{scope}
            \draw[blue](\x,-0.3,0)--(\x,-0.3,-1);
               \foreach \z in {0,1} {
\begin{scope}[shift={(.5,0,.5)}]
                    \drawWL{red}{\x,0,-\z}{\x,-0.3,-\z};
\end{scope}
                    \drawWL{blue}{\x,0,-\z}{\x,-0.3,-\z};
                 }
                }
               \foreach \z in {0,1} {
\begin{scope}[shift={(.5,0,.5)}]
                \drawWL{red} {1,0,-\z}{2,0,-\z};
\end{scope}
                \drawWL{blue} {1,0,-\z}{2,0,-\z};
                 }
   \foreach \z in {0,1} {
    \foreach \x in {1,2 } {
\begin{scope}[shift={(.5,0,.5)}]
    \node[red] at (\x,0,-\z) {$\bullet$};
    \node[red] at (\x,-0.3,-\z) {$\bullet$};
\end{scope}
    \node[blue] at (\x,0,-\z) {$\bullet$};
    \node[blue] at (\x,-0.3,-\z) {$\bullet$};
    }}
\end{tikzpicture}
}

\end{center}

 By the modularity of $\C$, the map induced by $LL$ is an isometry.
\end{proof}

\begin{proposition}\label{Prop:basis-independence}
Suppose $V$ is a Hilbert space and $\{\alpha_i\}$ is an ONB. Let $V'$ be the dual space of $V$. For $f\in V'$, a linear functional on $V$,
  \begin{align}
  r(f)=\sum_i \overline{f(\alpha_i)}\alpha_i
  \end{align}
  is independent of the choice of the basis.
\end{proposition}

\begin{proof}
It follows directly from definition.
\end{proof}

The map $r: V^*\to V$ is an anti-isometry which is well-known as the Riesz representation.
Therefore we obtain an anti-isometry $D:Conf(Lat')\to Conf(Lat)$ that we call the duality map:

\begin{definition}[duality maps]
We define
\begin{align*}
  \D_+(x)&=\sum_{x' \in B' } \overline{LL(x,x')}x',\\
  \D_-(x')&=\sum_{x \in B } \overline{LL(x,x')}x,
\end{align*}
where $B$ is an ONB of  $Conf(Lat)$ and $B'$ is an ONB of  $Conf(Lat')$.
\end{definition}

Therefore Theorem \ref{Thm:MN Self-duality} is equivalent to the following Proposition.

\begin{proposition}\label{Prop:Iso-isometry}
The map $\D_+$ is an anti-linear isometry from $Conf(Lat)$ to $Conf(Lat')$, and $D_-$ is its inverse.
\end{proposition}

\begin{definition}
We use $1_{n,m}$ to denote the trivial configuration whose $Z$-, $X$- ,$Y$-configurations are all 1.
We define
\begin{align}\label{Equ:mu}
\mu_{n,m}:&=\D_-(1_{n,m}).
\end{align}
\end{definition}

\begin{definition}
We define $L$ as a linear functional on $Conf(Lat)$ as $L(x)=LL(x,1_{n,m})$.
\end{definition}
Then
\begin{align}\label{Equ:L}
L(a_{\vec{j}}\otimes b_{\vec{i}}) &=\textcolor{black}{\mu^{\frac{(1-n)(m-1)}{2}}} \sqrt{d(X_{\vec{i},\vec{j}})}
\raisebox{-2cm}{
\begin{tikzpicture}
\myGT{3}{1};
            \foreach \x in {1,2,3,4} {
            \draw[blue](\x,-0.3,0)--(\x,-0.3,-2);
            \node at (\x,-.5,-1) {$b_{\x}$};
               \foreach \z in {0,1,2} {
                 }
                }
               \foreach \z in {0,1,2} {
                \drawWL{blue} {1,0,-\z}{4,0,-\z};
                \node at (2.5,.05,-\z) {$a_{\z}$};
                 }
   \foreach \z in {0,1,2} {
    \foreach \x in {1,2,3,4} {
    \node[blue] at (\x,0,-\z) {$\bullet$};
    \node[blue] at (\x,-0.3,-\z) {$\bullet$};
                    \draw[blue] (\x,-.15,-\z)--(\x,-0.3,-\z);
                    \draw[blue] (\x,0,-\z)--(\x,-0.15,-\z);
    }}
\end{tikzpicture}
}.
\end{align}
and
\begin{align}\label{Equ:mu L}
  \mu_{n,m}=\sum_{\alpha \in B } \overline{L(\alpha)} \alpha,
\end{align}
where $B$ be is an ONB of $Conf(Lat)$.

In \S \ref{Sec:Action on lattices}, we study the actions of rotations and reflections in $X$- and $Y$-directions on the lattices and the induced actions the configuration spaces.
In \S \ref{Sec:JW MTC}, \ref{Sec:PA of QM}, \ref{Sec:M self-duality}, we fix $m$ and study the structure of the configuration space for different $n$.
We prove that these configuration space admit the action of planar tangles (or operas) in the $X$-direction:
\begin{theorem}\label{Thm:MNtoPA}
For each $m\geq 1$, $\{\mathscr{S}_{n}=Conf(\C)_{m,n}\}_{n\in \mathbb{N}}$ is an unshaded subfactor planar algebra.
\end{theorem}
This defines the self-dual $m$-interval Jones-Wassermann subfactor. It is proved in Theorems \ref{Thm:JW subfactor} and \ref{Thm:selfdual}.
Moreover, the duality map defines the string Fourier transform (SFT) of the unshaded planar algebra.

\begin{remark}
If we fix $n$, instead of $m$, then all the results also work. So we also have the action of planar tangles on the configuration spaces in the $Y$-direction. Therefore the configuration spaces $\{Conf(\C)_{m,n}\}_{m,n\in \mathbb{N}}\}$ admit the action of planar tangles in two different directions.
\end{remark}

\begin{proposition}\label{Prop:0 map}
Let $B$ be an ONB of $\hom(\gamma,1)$ whose elements are $Y$-configurations. Let $1_\gamma$ be the canonical inclusion from $1$ to $\gamma$ and $b_1,b_2\in\hom(\gamma,1)$. Then

\begin{align*}
\delta^{-2}\sum_{b'\in B}d(X(b'))
\raisebox{-2cm}{
\begin{tikzpicture}
\myGT{3}{1};
            \foreach \x in {1,2} {
\begin{scope}[shift={(.5,0,.5)}]
            \draw[red](\x,-0.3,0)--(\x,-0.3,-2);
\end{scope}
            \draw[blue](\x,-0.3,0)--(\x,-0.3,-2);
            \node at (\x,-.5,-1) {$b_{\x}$};
               \foreach \z in {0,1,2} {
\begin{scope}[shift={(.5,0,.5)}]
                    \drawWL{red}{\x,0,-\z}{\x,-0.3,-\z};
\end{scope}
                    \drawWL{blue}{\x,0,-\z}{\x,-0.3,-\z};
                 }
                }
               \foreach \z in {0,1,2} {
\begin{scope}[shift={(.5,0,.5)}]
                \drawWL{red} {1,0,-\z}{2,0,-\z};
\end{scope}
                \drawWL{blue} {1,0,-\z}{2,0,-\z};
                 }
   \foreach \z in {0,1,2} {
    \foreach \x in {1,2} {
\begin{scope}[shift={(.5,0,.5)}]
    \node[red] at (\x,0,-\z) {$\bullet$};
    \node[red] at (\x,-0.3,-\z) {$\bullet$};
\end{scope}
    \node[blue] at (\x,0,-\z) {$\bullet$};
    \node[blue] at (\x,-0.3,-\z) {$\bullet$};
    }}
\node at (1.5,-.5,-.5) {$b'$};
\node at (2.6,-.5,-.5) {$\theta_1(b')$};
\end{tikzpicture}
}
&=\langle 1_\gamma^*, b_1 \rangle \langle 1_\gamma^*, b_2 \rangle
\end{align*}
\end{proposition}

\begin{proof}
Without loss of generality, we assume that $b_1$ and $b_2$ are unit vectors.
Note that if $X(b_1)\neq X(\theta_1(b_2))$, then both sides are zero.
We assume that $X(b_2)= X(\theta_1(b_1))$.

If a $Y$-configuration $b$ in $\hom(\gamma,1)$ is a unit vector, then
$$\dim \hom_{\C^m}(1, X(b) \otimes X(\theta(b)))=1.$$ So there is only one $X$-configuration with boundary $X(b) \otimes X(\theta(b))$ up to a scalar. Let $a_{b}$ be the the canonical inclusion from $1$ to $X(b) \otimes X(\theta(b))$ in $\C^m$. Let $C'=\{a'_{\vec{j}} \otimes b'_{\vec{i}}\}$ be an ONB of $Conf(Lat')$.
Applying Theorem \ref{Thm:MN Self-duality} for  $n=2$, we have that
\begin{align*}
&\langle 1_\gamma^*, b_1 \rangle \langle 1_\gamma^*, b_2 \rangle\\
=&\langle 1_{m,2}, a_{b_1} \otimes (b_1\otimes b_2) \rangle\\
=&\sum_{a'_{\vec{j}}\otimes b'_{\vec{i}} \in C'} LL( 1_{m,2}, a'_{\vec{j}}\otimes b'_{\vec{i}}) LL( a_{b_1} \otimes (b_1\otimes b_2),a'_{\vec{j}}\otimes b'_{\vec{i}} )\\
=&\sum_{b'\in B} LL(1_{m,2}, a'_{b'}\otimes (b'\otimes \theta_1(b') )) LL( a_{b_1} \otimes (b_1\otimes b_2),a'_{b'}\otimes (b'\otimes \theta_1(b') )) \\
=&\sum_{b'\in B} LL(1_{m,2}, a'_{b'}\otimes (b'\otimes \theta_1(b') )) LL( a_{b_1} \otimes (b_1\otimes b_2),a'_{b'}\otimes (b'\otimes \theta_1(b') )) \\
=&\delta^{-2} \sqrt{d(X(b_1))} \sum_{b'\in B} d(X(b'))
\scalebox{.5}{
\raisebox{-2cm}{
\begin{tikzpicture}
\myGT{3}{1};
            \foreach \x in {1,2} {
\begin{scope}[shift={(.5,0,.5)}]
            \draw[red](\x,-0.3,0)--(\x,-0.3,-2);
\end{scope}
            \draw[blue](\x,-0.3,0)--(\x,-0.3,-2);
            \node at (\x,-.5,-1) {$b_{\x}$};
               \foreach \z in {0,1,2} {
\begin{scope}[shift={(.5,0,.5)}]
                    \drawWL{red}{\x,0,-\z}{\x,-0.3,-\z};
\end{scope}
                    \drawWL{blue}{\x,0,-\z}{\x,-0.3,-\z};
                 }
                }
               \foreach \z in {0,1,2} {
\begin{scope}[shift={(.5,0,.5)}]
                \drawWL{red} {1,0,-\z}{2,0,-\z};
\end{scope}
                \drawWL{blue} {1,0,-\z}{2,0,-\z};
                 }
   \foreach \z in {0,1,2} {
    \foreach \x in {1,2} {
\begin{scope}[shift={(.5,0,.5)}]
    \node[red] at (\x,0,-\z) {$\bullet$};
    \node[red] at (\x,-0.3,-\z) {$\bullet$};
\end{scope}
    \node[blue] at (\x,0,-\z) {$\bullet$};
    \node[blue] at (\x,-0.3,-\z) {$\bullet$};
    }}
\node at (1.5,-.5,-.5) {$b'$};
\node at (2.6,-.5,-.5) {$\theta_1(b')$};
\end{tikzpicture}
}}
\end{align*}

If $X(b_1)\neq 1$, then both sides are zero. If $X(b_1)=1$, then $d(X(b_1))=1$ and the statement holds.
\end{proof}

If we switch $n$ and $m$ in Proposition \ref{Prop:0 map}, then we have obtain the following equivalent result:

\begin{proposition}\label{Prop:0 map'}
Take $\displaystyle \tilde{X}=\bigoplus_{X \in Irr} X$ and $1_{\tilde{X}}$ to be the conical inclusion from 1 to $\tilde{X}$. Then

\begin{align}\label{Equ:L}
\mu^{1-n} \sum_{X_{\vec{j}}\in Irr^n } d(X_{\vec{j}}) \sum_{\alpha \in ONB(X_{\vec{j}})}
\raisebox{-.8cm}{
\begin{tikzpicture}
\begin{scope}[xscale=.5,yscale=.5]
\myGT{3}{1};
\draw[red] (1,0,-1)--++(3,0,0);
            \foreach \x in {1,2,3}{
                \draw[white,WL] (\x+.5,-.5,-.5)--++(0,1,0);
                \draw[blue] (\x+.5,-.5,-.5)--++(0,1,0);
                }
\draw[white,WL] (1,0,0)--++(3,0,0);
\draw[red] (1,0,0)--++(3,0,0);
\node at (2.5,-.1,0){$\alpha^*$};
\node at (2.5,.1,-1){$\alpha$};
            \foreach \x in {1,2,3,4}{
                \draw[red] (\x,0,0)--(\x,0,-1);
                }
\end{scope}
\end{tikzpicture}
}
=&\quad\quad
\raisebox{-.8cm}{
\begin{tikzpicture}
\begin{scope}[xscale=.5,yscale=.5]
\myGT{3}{1};
\foreach \x in {1,2,3}{
                \draw[blue] (\x,-.5,-.5)--++(0,.3,0);
                \draw[blue] (\x,.5,-.5)--++(0,-.3,0);
\node[blue] at (\x,.2,-.5) {$\bullet$};
\node[blue] at (\x,-.2,-.5) {$\bullet$};
\node at (\x-.3,-.2,-.5) {$1_{\tilde{X}}$};
\node at (\x-.3,.2,-.5) {$1_{\tilde{X}}$};
                }
\end{scope}
\end{tikzpicture}
}.
\end{align}
\end{proposition}

\begin{proof}
  Taking inner product of both sides with an element in $\hom(\tilde{X}^n,\tilde{X}^n)$, the statement follows from Proposition \ref{Prop:0 map}.
\end{proof}

\begin{remark}
When $m=n=2$, Equation \eqref{Equ:L} is the killing relation.
\end{remark}

We prove Proposition \ref{Prop:0 map} and \ref{Prop:0 map'} using Theorem \ref{Thm:MN Self-duality}. Proposition \ref{Prop:0 map'} will be used in Lemma \ref{Lem:LLphi2}.
Then we prove Theorem \ref{Thm:MNtoPA} and Theorem \ref{Thm:MN Self-duality}.
We prove Theorems \ref{Thm:MN Self-duality} and \ref{Thm:MNtoPA} in the following order:
\begin{align*}
&\text{Theorem \ref{Thm:MN Self-duality} for $m=2$, $n=2$;}\\
\rightarrow&
\text{Theorem \ref{Thm:MNtoPA} for $m=2$;} \\
\rightarrow&
\text{Theorem \ref{Thm:MN Self-duality} for $m=2$, $n\geq 1$;} \\
\leftrightarrow&
\text{Theorem \ref{Thm:MN Self-duality} for $m\geq 1$, $n=2$;} \\
\rightarrow&
\text{Theorem \ref{Thm:MNtoPA} for $m\geq1$;} \\
\rightarrow&
\text{Theorem \ref{Thm:MN Self-duality} for $m\geq 1$, $n\geq 1$.} \\
\end{align*}
(When $m=1$, the configuration space $Conf(\C)_{m,n}$ is $\mathbb{C}$. The theorems are obvious.)

\section{Actions on configuration spaces}\label{Sec:Action on lattices}
\subsection{Automorphisms on the lattice}
Note that the lattice $\Z_m\times \Z_n$ is invariant under the following actions:
\begin{itemize}
\item The clockwise $2\pi/n$ rotation around the $Y$-direction $\rho_1$: $(i,j) \to (i-1,j)$.

\item The reflection in the $X$-direction $\theta_1$: $(i,j) \to (n+1-i,j)$.

\item The clockwise $2\pi/m$ rotation around the $X$-direction $\rho_2$: $(i,j) \to (i,j+1)$.

\item The reflection in the $Y$-direction $\theta_2$ $(i,j) \to (i,m-1-j)$.
\end{itemize}
Now let us define the induced action on the configuration space $Conf(\C)_{n,m}$.

For $k=1,2$, the induced actions on the $Z$-configurations are
\begin{align*}
\rho_k({X})_{i,j}&=X_{\rho_k^{-1}(i,j)},\\
\theta_k({X})_{i,j}&=\theta_{\C}(X_{\theta_k^{-1}(i,j)}).
\end{align*}

For an $X$-configuration $a_k$, we define

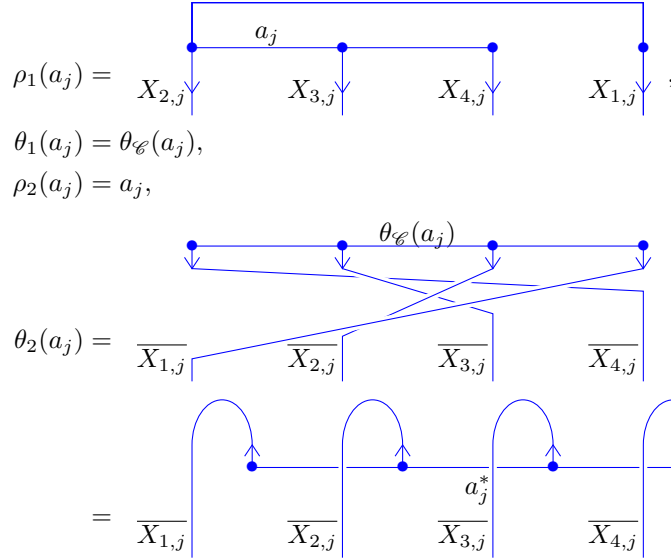
\begin{align*}
\rho_1(a_j)&=
\raisebox{-.5cm}{
\begin{tikzpicture}
\myGT{3}{1};
\node at (2.5,.05,0) {$a_j$};
\draw [blue] (2,0,0)--(4,0,0);
\foreach \x in {2,3,4} {
{
\drawWL{blue} {\x,0,0}{\x,-0.3,0};
\draw[blue] (2,0,0)--++(0,.2,0)--++(3,0,0)--++(0,-.2,0);
\draw[blue] (\x-.05,-0.2+.05,0)--(\x,-0.2,0)--(\x+.05,-0.2+.05,0);
\node at (\x-.2,-0.2,0) {$X_{\x,j}$};
\node[blue] at (\x,0,0) {$\bullet$};
}}
\foreach \x in {5} {
{
\drawWL{blue} {\x,0,0}{\x,-0.3,0};
\draw[blue] (2,0,0)--++(0,.2,0)--++(3,0,0)--++(0,-.5,0);
\draw[blue] (\x-.05,-0.2+.05,0)--(\x,-0.2,0)--(\x+.05,-0.2+.05,0);
\node at (\x-.2,-0.2,0) {$X_{1,j}$};
\node[blue] at (\x,0,0) {$\bullet$};
}}
\end{tikzpicture}
},\\
\theta_1(a_j)&=\theta_{\C}(a_j),\\
\rho_2(a_j)&=a_j,\\
\theta_2(a_j)&=
\raisebox{-.5cm}{
\begin{tikzpicture}
\myGT{3}{1};
\node at (2.5,.05,0) {$\theta_{\C}(a_j)$};
\draw [blue] (1,0,0)--(4,0,0);
\foreach \x in {1,2,3,4} {
{
\draw[white,WL] (\x,0,0)--++(0,-.1,0)--++(5-2*\x,-0.1*\x,0)--(5-\x,-.6,0);
\draw[blue] (\x,0,0)--++(0,-.1,0)--++(5-2*\x,-0.1*\x,0)--(5-\x,-.6,0);
\draw[blue] (\x-.05,-0.1+.05,0)--(\x,-0.1,0)--(\x+.05,-0.1+.05,0);
\node at (\x-.2,-0.5,0) {$\overline{X_{\x,j}}$};
\node[blue] at (\x,0,0) {$\bullet$};
}}
\end{tikzpicture}
}\\
&=\raisebox{-.5cm}{
\begin{tikzpicture}
\myGT{3}{1};
\node at (2.5,-.1,0) {$a_j^*$};
\draw [blue] (1,0,0)--(4,0,0);
\foreach \x in {1,2,3,4} {
{
\draw[white,WL] (\x,0,0)--++(0,.1,0) arc (0:180:.2) --++(0,-.5,0);
\draw[blue] (\x,0,0)--++(0,.1,0) arc (0:180:.2) --++(0,-.5,0);
\draw[blue] (\x-.05,0.1-.05,0)--(\x,0.1,0)--(\x+.05,0.1-.05,0);
\node at (\x-.6,-.3,0) {$\overline{X_{\x,j}}$};
\node[blue] at (\x,0,0) {$\bullet$};
}}
\end{tikzpicture}
}.
\end{align*}

For a $Y$-configuration $b_i$, we define

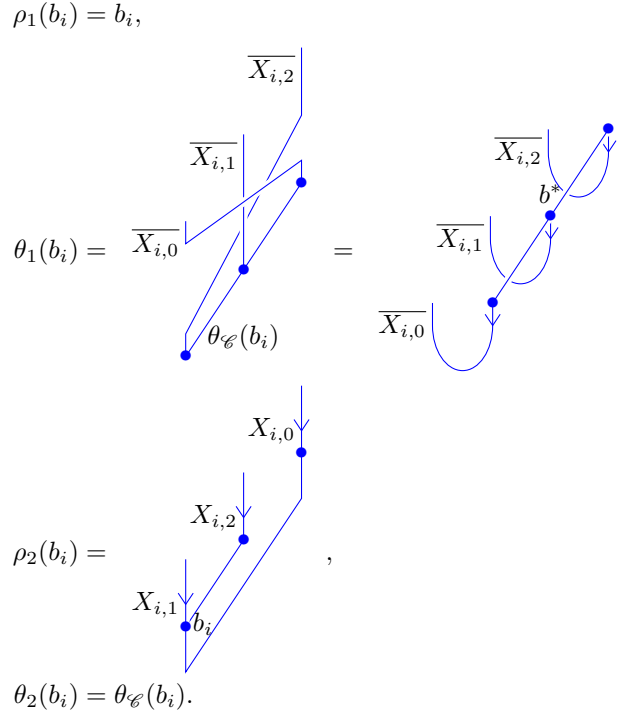
\begin{align*}
\rho_1(b_i)&=b_i,\\
\theta_1(b_i)&=
\raisebox{-1.5cm}{
\begin{tikzpicture}
\begin{scope}[xscale=1,yscale=1]
\myGT{3}{1};
\node at (0,-.9,-1) {$\theta_{\C}(b_i)$};
\foreach \z in {2,1,0}
{
\draw[white,WL] (0,0,-\z)--++(0,-.1,0)--++(0,-.1*\z,0)--(0,-0.5,\z-2)--(0,-0.6,\z-2);
\draw[blue] (0,0,-\z)--++(0,-.1,0)--++(0,-.1*\z,0)--(0,-0.5,\z-2)--(0,-0.6,\z-2);
}
\foreach \z in {2,1,0}
{
\node[blue] at (0,-0.6,-\z) {$\bullet$};
\node at (-.2,-0.1,-\z) {$\overline{X_{i,\z}}$};
}
\draw[blue] (0,-0.6,0)--(0,-0.6,-2);
\end{scope}
\end{tikzpicture}
}
=\raisebox{-1.5cm}{
\begin{tikzpicture}
\begin{scope}[xscale=1,yscale=1]
\myGT{3}{1};
\node at (0,.1,-1) {$b^*$};
\foreach \z in {2,1,0}
{
\draw[blue] (0,0,-\z) --++(0,-.1,0) arc (0:-180:.2 and .2) --++(0,.1,0);
\draw[blue] (-.05,-.05,-\z)--(0,-0.1,-\z)--(.05,-.05,-\z);
}
\draw[white,WL] (0,0,-2)--(0,0,0);
\draw[blue] (0,0,-2)--(0,0,0);
\foreach \z in {2,1,0}
{
\node[blue] at (0,0,-\z) {$\bullet$};
\node at (-0.6,-.1,-\z) {$\overline{X_{i,\z}}$};
}
\end{scope}
\end{tikzpicture}}
,\\
\rho_2(b_i)&=
\raisebox{-1.5cm}{
\begin{tikzpicture}
\myGT{3}{1};
\draw [blue] (0,-0.6,-2)--(0,-0.6,-1);
\draw [blue] (0,-0.6,-1)--++(0,-.2,0)--++(0,0,-2)--++(0,.2,0);
\node at (0,-.7,-1.3) {$b_i$};
\foreach \z in {1,2}
{
\drawWL{blue} {0,-.3,-\z}{0,-0.6,-\z};
\draw[blue] (-.05,-0.5+.05,-\z)--(0,-0.5,-\z)--(.05,-0.5+.05,-\z);
\node at (-.2,-0.5,-\z) {$X_{i,\z}$};
\node[blue] at (0,-0.6,-\z) {$\bullet$};
}
\foreach \z in {3}
{
\drawWL{blue} {0,-.3,-\z}{0,-0.6,-\z};
\draw[blue] (-.05,-0.5+.05,-\z)--(0,-0.5,-\z)--(.05,-0.5+.05,-\z);
\node at (-.2,-0.5,-\z) {$X_{i,0}$};
\node[blue] at (0,-0.6,-\z) {$\bullet$};
}
\end{tikzpicture}
},\\
\theta_2(b_i)&=\theta_{\C}(b_i).\\
\end{align*}

\begin{definition}
For a configuration $a_{\vec{j}}\otimes b_{\vec{i}}$, we define
\begin{align*}
  \rho_1(a_{\vec{j}}\otimes b_{\vec{i}})&=(\rho_1(a_{0})\otimes \cdots \otimes \rho_1(a_{m-1})) \otimes (b_{2} \otimes \cdots \otimes b_{n} \otimes b_{1}), \\
  \theta_1(a_{\vec{j}}\otimes b_{\vec{i}})&=(\theta_1(a_{0})\otimes \cdots \otimes \theta_1(a_{m-1})) \otimes (\theta_1(b_{n}) \otimes \cdots \otimes \theta_{1}(b_{1})),\\
  \rho_2(a_{\vec{j}}\otimes b_{\vec{i}})&=(a_{m-1} \otimes a_{1} \otimes \cdots \otimes a_{m-2}) \otimes (\rho_2(b_{1}) \otimes \cdots \otimes \rho_2(b_{n}) ), \\
  \theta_2(a_{\vec{j}}\otimes b_{\vec{i}})&=(\theta_2(a_{m-1})\otimes \cdots \otimes \theta_2(a_{0})) \otimes (\theta_2(b_{1}) \otimes \cdots \otimes \theta_{2}(b_{n})).
\end{align*}
The actions on the configuration space are defined by their linear or anti-linear extensions.
\end{definition}

Note that this definition coincide with the geometric actions on the configurations in Fig.~\ref{Fig:configuration}. Therefore their relations also hold on the configuration space.
\begin{proposition}\label{Prop:Rho Theta commute}
On the configuration space, $\rho_1$ and $\theta_1$ commute with $\rho_2$ and $\theta_2$, and for $k=1,2$,
$\rho_k\theta_k=\theta_k\rho_k^{-1}$, $\rho_k^m=1$, $\theta_k^2=1$.
\end{proposition}

\subsection{Automorphisms on the dual pair of lattices}
Similarly we also define the four actions on the dual lattice $Lat'$ and the configuration space $Conf(Lat')$.

\begin{proposition}\label{Prop:LLrhotheta}
For $x\in Conf(Lat)$ and $x'\in Conf(Lat')$, we have that
\begin{align}
LL(x,x')&=LL(\rho_k(x),\rho_k(x')), \quad k=1,2 \label{Equ:LLrho}\\
\overline{LL(x,x')}&=LL(\theta_1(x),\rho_1\theta_1(x')), \label{Equ:LLtheta1}\\
\overline{LL(x,x')}&=LL(\rho_1\theta_2(x),\theta_2(x')). \label{Equ:LLtheta2}
\end{align}
\end{proposition}

\begin{proof}
It is enough to prove the case $x=a_{\vec{j}}\otimes b_{\vec{i}}$, $x'=a'_{\vec{j}}\otimes b'_{\vec{i}}$.

Recall that $LL$ is defined by a closed diagram in the 3D space as shown in Equation \eqref{Equ:bilinear form}. Applying the rotation on the 3D diagram, we obtain Equation \eqref{Equ:LLrho}.

If we consider the $3D$ diagram as an element in $\C$, then we have that
\begin{align*}
&\overline{LL(a_{\vec{j}}\otimes b_{\vec{i}}, a'_{\vec{j}}\otimes b'_{\vec{i}})} \\
=& \mu^{\frac{(1-n)(m-1)}{2}} \sqrt{d(X_{\vec{i},\vec{j}})d(X'_{\vec{i},\vec{j}})}\\
&\raisebox{0cm}
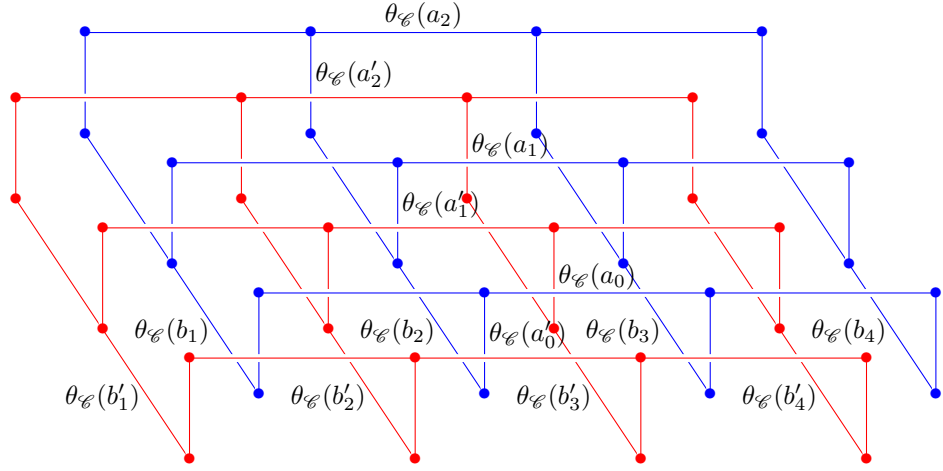
{
\begin{tikzpicture}
\begin{scope}[xscale=-1.5,yscale=1.5]
\myGT{3}{1};
     \foreach \x in {1,2,3,4} {
\begin{scope}[shift={(.5,0,.5)}]
            \draw[red](\x,-0.3,0)--(\x,-0.3,-2);
            \node at (5-\x,-.5,-1) {$\theta_{\C}(b'_{\x})$};
\end{scope}
            \draw[blue](\x,-0.3,0)--(\x,-0.3,-2);
            \node at (5-\x,-.5,-1) {$\theta_{\C}(b_{\x})$};
               \foreach \z in {0,1,2} {
\begin{scope}[shift={(.5,0,.5)}]
                    \drawWL{red}{\x,0,-\z}{\x,-0.3,-\z};
\end{scope}
                    \drawWL{blue}{\x,0,-\z}{\x,-0.3,-\z};
                 }
                }
               \foreach \z in {0,1,2} {
\begin{scope}[shift={(.5,0,.5)}]
                \drawWL{red} {1,0,-\z}{4,0,-\z};
                \node at (2.5,.07,-\z) {$\theta_{\C}(a'_{\z})$};
\end{scope}
                \drawWL{blue} {1,0,-\z}{4,0,-\z};
                \node at (2.5,.05,-\z) {$\theta_{\C}(a_{\z})$};
                 }
   \foreach \z in {0,1,2} {
    \foreach \x in {1,2,3,4} {
\begin{scope}[shift={(.5,0,.5)}]
    \node[red] at (\x,0,-\z) {$\bullet$};
    \node[red] at (\x,-0.3,-\z) {$\bullet$};
\end{scope}
    \node[blue] at (\x,0,-\z) {$\bullet$};
    \node[blue] at (\x,-0.3,-\z) {$\bullet$};
    }}
\end{scope}
\end{tikzpicture}
}\\
=& \mu^{\frac{(1-n)(m-1)}{2}} \sqrt{d(X_{\vec{i},\vec{j}})d(X'_{\vec{i},\vec{j}})}\\
&\raisebox{0cm}
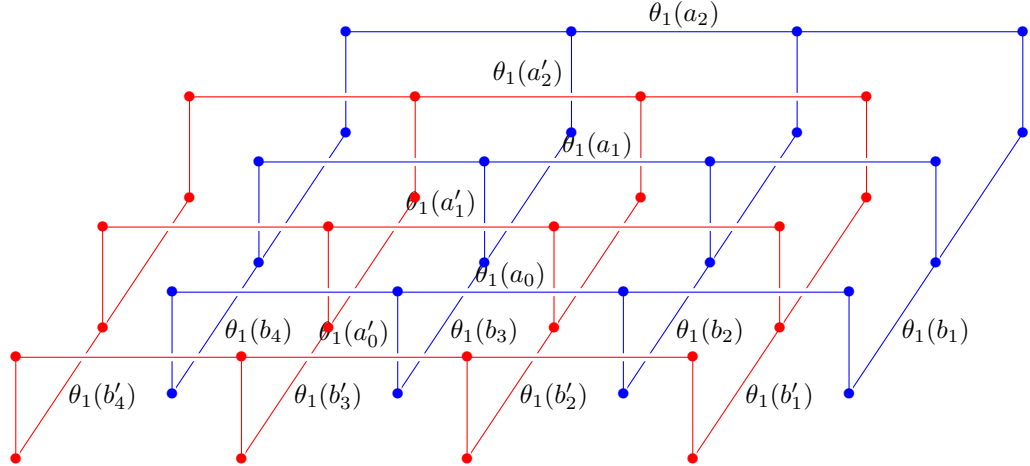
{
\begin{tikzpicture}
\begin{scope}[xscale=1.5,yscale=1.5]
\myGT{3}{1};
            \foreach \x in {1,2,3,4} {
\begin{scope}[shift={(.5,0,.5)}]
            \draw[red](\x-1,-0.3,0)--(\x-1,-0.3,-2);
            \node at (5-\x-1,-.5,-1) {$\theta_1(b'_{\x})$};
\end{scope}
            \draw[blue](\x,-0.3,0)--(\x,-0.3,-2);
            \node at (5-\x,-.5,-1) {$\theta_1(b_{\x})$};
               \foreach \z in {0,1,2} {
\begin{scope}[shift={(.5,0,.5)}]
                    \drawWL{red}{\x-1,0,-\z}{\x-1,-0.3,-\z};
\end{scope}
                    \drawWL{blue}{\x,0,-\z}{\x,-0.3,-\z};
                 }
                }
               \foreach \z in {0,1,2} {
\begin{scope}[shift={(.5,0,.5)}]
                \drawWL{red} {0,0,-\z}{3,0,-\z};
                \node at (2.5-1,.07,-\z) {$\theta_1(a'_{\z})$};
\end{scope}
                \drawWL{blue} {1,0,-\z}{4,0,-\z};
                \node at (2.5,.05,-\z) {$\theta_1(a_{\z})$};
                 }
   \foreach \z in {0,1,2} {
    \foreach \x in {1,2,3,4} {
\begin{scope}[shift={(.5,0,.5)}]
    \node[red] at (\x-1,0,-\z) {$\bullet$};
    \node[red] at (\x-1,-0.3,-\z) {$\bullet$};
\end{scope}
    \node[blue] at (\x,0,-\z) {$\bullet$};
    \node[blue] at (\x,-0.3,-\z) {$\bullet$};
    }}
\end{scope}
\end{tikzpicture}
}\\
&=LL(\theta_1(x),\rho_1\theta_1(x'))
\end{align*}

The proof of Equation \eqref{Equ:LLtheta2} is similar.
\end{proof}

By definitions,  $\rho_k$, $\theta_k$, $k=1,2$, preserve $1_{n,m}$.
Take $x'=1_{n,m}$ in Proposition \ref{Prop:LLrhotheta}, we obtain that

\begin{proposition}\label{Prop:L*linear}
For any $x$ in $Conf(Lat)$, $k=1,2$,
\begin{align*}
  L(\rho_k(x))&=L(x),\\
  L(\theta_k(x))&=\overline{L(x)}.
\end{align*}
\end{proposition}

\begin{proposition}\label{Prop:Rho Theta preserve}
The four actions $\rho_k$, $\theta_k$, $k=1,2$, preserve $\mu_{n,m}$.
\end{proposition}

\begin{proof}
The statement follows from Proposition \ref{Prop:L*linear} and the definition of $\mu_{n,m}$ in Equation \eqref{Equ:mu L}.
\end{proof}

\section{Jones-Wassermann subfactors for MTC}\label{Sec:JW MTC}

\subsection{Identification}

Suppose $\mathscr{C}$ is a unitary modular tensor category.
Let $\mathscr{C}^m$ be the $mth$ tensor power of $\mathscr{C}$.
Let $Irr_m$ be the set of simple objects of $\mathscr{C}^m$. We can represent a simple object in $Irr_m$ as $\vec{X}:=X_0\otimes \cdots \otimes X_{m-1}$ for some simple objects $X_j$ in $\mathscr{C}$. Let $d(\vec{X})$ be the quantum dimension of the object $\vec{X}$. Then $\displaystyle d(\vec{X})=\prod_{j=0}^{m-1} d(X_j)$.

Take
\begin{align}
\gamma=\gamma_m&=\bigoplus_{\vec{X}} N_{\vec{X}} \vec{X},
\end{align}
where $N_{\vec{X}}=\dim \hom_\mathscr{C} (\vec{X},1)$.
Recall that $\mu$ is the global dimension of $\C$.

\begin{proposition}\label{Prop:dim gamma}
For $m\geq 1$, $d(\gamma)=\mu^{m-1}$.
\end{proposition}

\begin{proof}
It is obvious for $m=1$. When $m\geq2$,
\begin{align*}
d(\gamma)&=\sum_{\vec{X}\in Irr_m} N_{\vec{X}} d(\vec{X})\\
&=  \sum_{\vec{X}\in Irr_{m-1}, Y\in Irr_{1}} \dim \hom_{\C}(\vec{X},Y)  d(\vec{X})d(Y)   &&\text{by Frobenius reciprocity}\\
&=\sum_{\vec{X}\in Irr_{m-1}}   d(\vec{X})^2\\
&=\mu^{m-1}.
\end{align*}
\end{proof}

\begin{notation}
Take $\delta=\mu^{\frac{m-1}{2}}$.
\end{notation}

For each $\vec{X}$, let $ONB^*_{\C}(\vec{X})$ be an ONB of $\hom_{\C}(\vec{X},1)$.
Then we can use the basis to represent the multiplicity of $\vec{X}$ in $\gamma$.
\begin{align}
\gamma&=\bigoplus_{\vec{X}\in Irr_m, b\in ONB^*_{\C}(\vec{X})} \vec{X}(b),
\end{align}
where $\vec{X}(b)=\vec{X}$.

The representation is covariant with respect to the choice of the ONB:
For an object $Y$ in $\C^m$ and a morphism $y \in  \hom_{\C^m} (Y, N_{\vec{X}} \vec{X})$, we take two ONB $B(1)$, $B(2)$ of $\hom_{\C}(\vec{X},1)$. Then we obtain two representations
  \begin{align*}
  y&=\bigoplus_{b_1\in B(1)} y(b_1), && y(b_1)\in \hom (Y,\vec{X}(b_1)), \\
  y&=\bigoplus_{b_2\in B(2)} y(b_2), && y(b_2)\in \hom (Y,\vec{X}(b_2)).
  \end{align*}
The representation is covariant means that
   $$y(b)=\sum_{b'} <b',b> y(b').$$

Note that
\begin{align*}
\hom_{\C^m}(1,\gamma^n)
&=\bigoplus_{X_{\vec{i},\vec{j}}\in Irr^{nm}} \bigoplus_{b_i \in ONB^*_{\C}(X_{i,\vec{j}}), 1\leq i\leq n} \hom_{\C^m}(1,  X_{\vec{i},\vec{j}}(b_{\vec{i}})).
\end{align*}
We call it the \emph{$n$-box space} of $\gamma$.
For $a_{j} \in \hom_{\C}(1, X_{\vec{i},j})$, $0\leq j \leq m-1$, we have $a_{\vec{j}}(b_{\vec{i}}) \in \hom_{\C^m}(1,  X_{\vec{i},\vec{j}}(b_{\vec{i}}))$.

\begin{definition}
We define a map $\Phi:\hom_{\C^m}(1,\gamma^n) \to Conf(\C)_{n,m}$ as a linear extension of
\begin{align*}
  \Phi(a_{\vec{j}}(b_{\vec{i}}))=a_{\vec{j}}\otimes b_{\vec{i}}.
\end{align*}
\end{definition}
The definition is independent of the choice of the ONB $b_{\vec{i}}$, since the representation is covariant.
Moreover,
\begin{align*}
 \langle a_{\vec{j}}(b_{\vec{i}}),c_{\vec{j}}(d_{\vec{i}})\rangle =\langle a_{\vec{j}},c_{\vec{j}}\rangle \langle b_{\vec{i}},d_{\vec{i}}\rangle =\langle a_{\vec{j}}\otimes b_{\vec{i}},c_{\vec{j}}\otimes d_{\vec{i}}\rangle .
\end{align*}
So $\Phi$ is an isometry. Therefore we can identify the vectors in the two Hilbert spaces $\hom_{\C^m}(1,\gamma^n)$ and $Conf(\C)_{n,m}$.
We simply use the notation $\sum a_{\vec{j}}(b_{\vec{i}})$ to represent an element in $\hom_{\C^m}(1,\gamma^n)$.

\begin{definition}
Induced by the isometry $\Phi$, the four actions $\rho_k$, $\theta_k$, $k=1,2$, and the contractions $\wedge_k$, $k\geq 0$, are also defined on $\hom(1,\gamma^n)$, still denoted by  $\rho_k$, $\theta_k$.
\end{definition}

Recall that the multiplicity of $X_{-,\vec{j}}$ in $\gamma$ is represented by $b$ in $ONB^*(X_{-,\vec{j}})$.
We need an anti-isometric involution on $ONB^*(X_{-,\vec{j}})$ to specify the dual of $X_{-,\vec{j}}(b)$.
To be compatible with the geometric interpretation of the configuration in the 3D space, we define the dual by $\theta_1$:

\begin{definition}\label{Def:thetaX}
For an object $X_{-,\vec{j}}(b)$, we define its dual object as  $\overline{X_{-,\vec{j}}}(\theta_1(b))$.
\end{definition}
Note that $\theta_1^2(b)=b$, thus $\overline{\overline{X_{-,\vec{j}}(b)}}=X_{-,\vec{j}}(b)$.
By Frobenius reciprocity, the modular conjugation on $\C^m$ is given by $\theta_1$.

The element in $\hom_{\C^m}(1,\gamma^n)$ is usually represented by a diagram on the 2D plane.
To be compatible with the isometry $\Phi$, we simplify the 3D pictures for configurations by their projections on the plane $Y=0$ as follows:

The configuration in Fig.~\ref{Fig:configuration} is simplified as the following notation:
$$
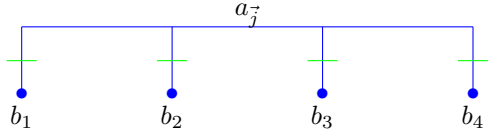
\raisebox{0cm}{
\begin{tikzpicture}
\myGT{3}{1};
            \foreach \x in {1,2,3,4} {
            \draw[blue](\x,-0.3,0)--(\x,0,0);
            \node at (\x,-.4,0) {$b_{\x}$};
                \node[blue] at (\x,-0.3,0) {$\bullet$};
                \draw[green] (\x-.1,-0.15,0)--++(.2,0,0);
                }
                \draw[blue] (1,0,0)--(4,0,0);
                \node at (2.5,.05,0) {$a_{\vec{j}}$};
\end{tikzpicture}
}.$$

Induced by the isometry $\Phi$, $L$ becomes a linear functional on $\hom(1,\gamma^n)$,
\begin{align*}
L(a_{\vec{j}}(b_{\vec{i}})):=
L(a_{\vec{j}}\otimes b_{\vec{i}})=&\delta^{1-n} \sqrt{d(X_{\vec{i},\vec{j}})}
\raisebox{-1cm}{
\begin{tikzpicture}
\myGT{3}{1};
            \foreach \x in {1,2,3,4} {
            \draw[blue](\x,-0.3,0)--(\x,0,0);
            \node at (\x,-.4,0) {$b_{\x}$};
                \node[blue] at (\x,-0.3,0) {$\bullet$};
                }
                \draw[blue] (1,0,0)--(4,0,0);
                \node at (2.5,.05,0) {$a_{\vec{j}}$};
\end{tikzpicture}
},
\end{align*}
where we simplify the diagram in Equation \eqref{Equ:L} by its projection on the plane $Y=0$.

\subsection{Contractions}\label{Sec:Contraction}
The multiplication on $\C^m$ defines a map from $\hom(\gamma^{n},\gamma^{k}) \otimes \hom(\gamma^{k},\gamma^{l})$ to $\hom(\gamma^{n},\gamma^{l})$.
Applying Frobenius reciprocity, we obtain a contraction $\wedge_k: \hom(1,\gamma^{n+k}) \otimes \hom(1,\gamma^{k+l}) \to \hom(1,\gamma^{n+l})$.
Then  $\wedge_k$ is also defined on the configuration spaces induced by $\Phi$,. We give the definition in detail here.

\begin{remark}
The notation $\wedge_k$ comes from the graded multiplication in \cite{GJS}.
\end{remark}

Suppose $X_{\vec{i},\vec{j}}$, $Y_{\vec{i},\vec{j}}$, $Z_{\vec{i},\vec{j}}$ are $Z$-configurations of size $n\times m$, $\ell\times m$, $k\times m$.
For $X$-configurations $a_{\vec{j}} \in \hom_{\C}(1,X_{\vec{i},\vec{j}}\otimes Z_{\vec{i},\vec{j}})$
and $c_{\vec{j}} \hom(1,\theta_1(Z_{\vec{i},\vec{j}}) \otimes Y_{\vec{i},\vec{j}})$, we define the $k$-string contraction, $k\geq0$, as
\begin{align*}
a_{j}\wedge_k c_{j}:&=
\raisebox{-.3cm}{
\begin{tikzpicture}
\begin{scope}[xscale=.75]
\myGT{3}{1};
            \foreach \x in {1,2,3,6,7,8} {
            \draw[blue](\x,-0.4,0)--(\x,0,0);
                }
            \foreach \x in {4,5} {
            \draw[blue](\x,-0.2,0)--(\x,0,0);
                }
            \node at (1-.2,-0.2,0) {$X_{1,j}$};
            \node at (1.5-.2,-0.2,0) {$\cdots$};
            \node at (2-.2,-0.2,0) {$X_{n,j}$};
            \node at (3-.2,-0.2,0) {$Z_{1}$};
            \node at (3.5-.2,-0.2,0) {$\cdots$};
            \node at (4-.2,-0.1,0) {$Z_{k}$};
            \node at (5-.2,-0.1,0) {$\overline{Z_{k}}$};
            \node at (5.5-.2,-0.2,0) {$\cdots$};
            \node at (6-.2,-0.2,0) {$\overline{Z_{1}}$};
            \node at (7-.2,-0.2,0) {$Y_{1,j}$};
            \node at (7.5-.2,-0.2,0) {$\cdots$};
            \node at (8-.2,-0.2,0) {$Y_{\ell,j}$};
                \draw[blue] (1,0,0)--(4,0,0);
                \draw[blue] (5,0,0)--(8,0,0);
                \node at (2.5,.07,0) {$a_{j}$};
                \node at (6.5,.07,0) {$c_{j}$};
                \draw[blue] (3,-0.4,0)--++(3,0,0);
                \draw[blue] (4,-0.2,0)--++(1,0,0);
\end{scope}
\end{tikzpicture}}
\end{align*}

Moreover, $a_{\vec{j}}\wedge_k c_{\vec{j}}:=(a_{1}\wedge_k c_{1}) \otimes \cdots \otimes (a_{m}\wedge_k c_{m})$.
Suppose $b_{\vec{i}} \in \hom_{\C}(1,X_{\vec{i},\vec{j}}\otimes Z_{\vec{i},\vec{j}})$ and $d_{\vec{i}}\hom(1,\theta_1(Z_{\vec{i},\vec{j}}) \otimes Y_{\vec{i},\vec{j}})$ are $Y$-configurations.
We define the $k$-string contraction $\wedge_k$ on the configurations $a_{\vec{j}}\otimes b_{\vec{i}}$ and $c_{\vec{j}}\otimes d_{\vec{i}}$ as
\begin{align*}
(a_{\vec{j}}\otimes b_{\vec{i}}) \wedge_k (c_{\vec{j}}\otimes d_{\vec{i}})
&=
\prod_{s=1}^{k} \langle \theta_1(d_{k+1-s}), b_{n+s} \rangle
(a_{\vec{j}}\wedge_k c_{\vec{j}})\otimes (\bigotimes_{i=1}^{n} b_i \otimes  \bigotimes_{i=k+1}^{k+\ell}  d_{i}).
\end{align*}
We define $\wedge_k: Conf(\C)_{n+k,m} \otimes Conf(\C)_{k+\ell,m} \to Conf(\C)_{k+\ell,m}$.
by a linear extension on the configuration spaces.

Therefore we can identity $Conf(\C)_{k+\ell,m}$ as operators from $Conf(\C)_{k,m}$ to $Conf(\C)_{\ell,m}$ corresponding to Frobenius reciprocity.
Moreover, the composition of these operators is associative.

\begin{proposition}\label{Prop:Rho Theta Contraction}
Recall that $\rho_2$ and $\theta_2$ are actions in the $Y$-directions. They commute with the contraction $\wedge_k$ in the $X$-direction.
\end{proposition}

\begin{proof}
Recall that $\rho_2$ is an isometry and it commutes with $\theta_1$ by Proposition \ref{Prop:Rho Theta commute}, so
\begin{align*}
&\rho_2(a_{\vec{j}}\otimes b_{\vec{i}}) \wedge_k \rho_2(c_{\vec{j}}\otimes d_{\vec{i}})\\
=&
\prod_{s=1}^{k} \langle \theta_1\rho_2(d_{k+1-s}),\rho_2(b_{n+s}) \rangle
(\rho_2(a_{\vec{j}}) \wedge_k \rho_2(c_{\vec{j}}))\otimes (\bigotimes_{i=1}^{n} \rho_2(b_i) \otimes  \bigotimes_{i=k+1}^{k+\ell}  \rho_2(d_{i}))\\
=&
\prod_{s=1}^{k} \langle \theta_1(d_{k+1-s}),(b_{n+s}) \rangle
\rho_2(a_{\vec{j}} \wedge_k c_{\vec{j}})\otimes (\bigotimes_{i=1}^{n} \rho_2(b_i) \otimes  \bigotimes_{i=k+1}^{k+\ell}  \rho_2(d_{i}))\\
=&\rho_2((a_{\vec{j}}\otimes b_{\vec{i}}) \wedge_k (c_{\vec{j}}\otimes d_{\vec{i}})).
\end{align*}

The proof for $\theta_2$ is similar.
\end{proof}

\begin{lemma}\label{Lem:contraction rule}
When $k=1$, we have

\begin{align*}
&
\raisebox{-.3cm}{
\begin{tikzpicture}
\myGT{3}{1};
            \foreach \x in {1,2,5,6} {
            \draw[blue](\x,-.3,0)--(\x,0,0);
                }
                \foreach \x in {1,2,5,6} {
            \node[blue] at (\x,-.3,0) {$\bullet$};
                }
                \node at (1,-.4,0) {$b_{1}$};
                \node at (2,-.4,0) {$b_{n}$};
                \node at (5,-.4,0) {$d_{1+1}$};
                \node at (6,-.4,0) {$d_{1+\ell}$};
            \node at (1.5-.2,-0.2,0) {$\cdots$};
            \node at (5.5-.2,-0.2,0) {$\cdots$};
                \draw[blue] (1,0,0)--(6,0,0);
                \node at (2,.07,0) {$a_{\vec{j}}$};
                \node at (5,.07,0) {$c_{\vec{j}}$};
\end{tikzpicture}
}\\
=\sum_{b\in ONB(Z_{\vec{j}})}&
\raisebox{-1.cm}{
\begin{tikzpicture}
\myGT{3}{1};
            \foreach \x in {1,2,3,4,5,6} {
            \draw[blue](\x,-.3,0)--(\x,0,0);
                }
                \foreach \x in {1,2,3,4,5,6} {
            \node[blue] at (\x,-.3,0) {$\bullet$};
                }
                \node at (1,-.4,0) {$b_{1}$};
                \node at (2,-.4,0) {$b_{n}$};
                \node at (3,-.4,0) {$b$};
                \node at (4,-.4,0) {$\theta_1(b)$};
                \node at (5,-.4,0) {$d_{1+1}$};
                \node at (6,-.4,0) {$d_{1+\ell}$};
            \node at (1.5-.2,-0.2,0) {$\cdots$};
            \node at (5.5-.2,-0.2,0) {$\cdots$};
                \draw[blue] (1,0,0)--(3,0,0);
                \draw[blue] (4,0,0)--(6,0,0);
                \node at (2,.07,0) {$a_{\vec{j}}$};
                \node at (5,.07,0) {$c_{\vec{j}}$};
\end{tikzpicture}
}.
\end{align*}
\end{lemma}

\begin{proof}
We apply Equation \eqref{Equ:resolution of identity}, the resolution of identity in $\C$, to $1_{Z_{\vec{j}}}$ in the first diagram. Only the component equivalent to 1 remains non-zero, since the first diagram has no boundary on the left. This component gives the second diagram.
\end{proof}

\subsection{Frobenius algebras}

\begin{definition}
We define $\mu_n=\Phi^{-1}(\mu_{n,m})$, for $n\geq 1$, and $\mu_0=1$.
Then
\begin{align*}
  \mu_{n}=\sum_{\alpha \in B } \overline{L(\alpha)} \alpha,
\end{align*}
where $B$ be is an ONB of $\hom(1,\gamma^n)$.
\end{definition}
Moreover,
$\mu_1$ is the canonical inclusion from $1$ to $\gamma$, and
$\mu_2$ is the canonical inclusion from $1$ to $\gamma\otimes \overline{\gamma}$ which defines the dual of objects.
By Proposition \ref{Prop:Rho Theta preserve}, for $k=1,2$,
\begin{align}
\rho_k(\mu_n)&=\mu_n, \label{Equ:rho-mu}\\
\theta_k(\mu_n)&=\mu_n. \label{Equ:theta-mu}
\end{align}
Let us prove that $\mu_3$ defines a Frobenius algebra.

\begin{lemma}\label{Lem:mu composition}
For $n,\ell\geq1$,
\begin{align*}
\mu_{n}\wedge_1\mu_{\ell}=\delta^{-1} \mu_{n+\ell-2}.
\end{align*}
\end{lemma}

\begin{proof}
Suppose $X_{\vec{i},\vec{j}}$, $Y_{\vec{i},\vec{j}}$, $Z_{\vec{j}}$ are $Z$-configurations of size $n\times m$, $\ell\times m$, $1\times m$.
Note that
$$\{\sqrt{d(Z_j)} a_{j}\wedge_1 c_{j} ~|~ a_j \in ONB(X_{\vec{i},j}\otimes Z_{j}), ~ c_j\in ONB(\overline{Z_{j}}\otimes Y_{\vec{i},j}) \} $$
defines an $ONB(X_{\vec{i},j}\otimes Y_{\vec{i},j})$.
Take $b_{\vec{i}}=b_{1} \otimes \cdots \otimes b_{n}$ and
$d_{\vec{i}}=d_{1} \otimes \cdots \otimes d_{\ell}$, where $b_{i}\in ONB(X_{i,\vec{j}})$ and $d_{i}\in ONB(Y_{i,\vec{j}})$. By Lemma \ref{Lem:contraction rule}, we have
\begin{align*}
&\overline{L(\sqrt{d(Z_j)} a_{\vec{j}}\wedge_1 c_{\vec{j}}(b_{\vec{i}}\otimes d_{\vec{i}}) )} \sqrt{d(Z_j)} a_{\vec{j}}\wedge_1 c_{\vec{j}}(b_{\vec{i}}\otimes d_{\vec{i}})\\
=&\delta \sum_{b\in ONB^*(Z_{\vec{\vec{j}}})} \overline{L( a_{\vec{j}}(b_{\vec{i}}\otimes b) ) L(c_{\vec{j}} (\theta_1(b) \otimes d_{\vec{i}}) )}  a_{\vec{j}}\wedge_1 c_{\vec{j}}(b_{\vec{i}}\otimes d_{\vec{i}})  \\
=&\delta \left(\sum_{b\in ONB^*(Z_{\vec{\vec{j}}})} \overline{L( a_{\vec{j}}(b_{\vec{i}}\otimes b))}   a_{\vec{j}}(b_{\vec{i}}\otimes b) \right) \wedge_1 \left( \sum_{b\in ONB^*(Z_{\vec{\vec{j}}})} \overline{L(c_{\vec{j}} (\theta_1(b) \otimes d_{\vec{i}}) )} c_{\vec{j}}(\theta_1(b)\otimes d_{\vec{i}}) \right) .
\end{align*}
Sum over $a_{\vec{j}}(b_{\vec{i}})$, $c_{\vec{j}}(d_{\vec{i}})$, we have
\begin{align*}
\mu_{n+\ell-2}=\delta \mu_{n}\wedge_1\mu_{\ell}.
\end{align*}

\end{proof}

\begin{theorem}\label{Thm:JW subfactor}
By Frobenius reciprocity, we can identify $\mu_3$ as a morphism in $\hom(\gamma,\gamma^2)$ or $\hom(\gamma^2,\gamma)$. Then $(\gamma,\mu_{3})$ is a Frobenius algebra in $\C^m$.
\end{theorem}

\begin{proof}
It follows from Equations \eqref{Equ:rho-mu}, \eqref{Equ:theta-mu} and Lemma \ref{Lem:mu composition}.
\end{proof}

We call the subfactor associated with the Frobenius algebra $(\gamma,\mu_{3})$ the \emph{$m$-interval Jones-Wassermann subfactor}.
The modularity is not used in the construction of the Jones-Wassermann subfactors, but it is crucial in the proof of the self-duality of the Jones-Wassermann subfactor.
The formula of $(\gamma,\mu_3)$ in terms of the 3D configuration is intuitive in the proof of self-duality.

We can also derive this Frobenius algebra through the tensor functor $Fun$ from $\C^m$ to $\C$.
(1) The adjoint functor of $Fun$ sends the trivial Frobenius algebra in $\C$ to a Frobenius algebra in $\C^m$ which is our $(\gamma,\mu_3)$.
(2) The functor $Fun$ defines an inclusion from $\hom_{\C^m}(\tilde{X}^{mn})$ to $\hom_{\C}(\tilde{X}^{mn})$. The inductive limit of this inclusion for $n\to \infty$ defines a subfactor, which was studied by Erlijman and Wenzl in \cite{ErlWen07}. The corresponding Frobenius algebra is $(\gamma,\mu_3)$.

The Frobenius algebra $(\gamma,\mu_3)$ defines a $\gamma-\gamma$ bimodule category induced by $\C^m$. It is a unitary fusion category, called the dual of $\C^m$ with respect to $(\gamma,\mu_3$.
When $m=2$, the dual of $\C^2$ is known as the quantum double of $\C$.
For a general $m$, we call the dual of $\C^m$ with respect to the Frobenius algebra $(\gamma,\mu_3)$ the {\it quantum $m$-party, or quantum multiparty,} of $\C$.

\section{The string Fourier transform on planar algebras}\label{Sec:PA of QM}

Once we obtain a Frobenius algebra $(\gamma,\mu_3)$, we can define a subfactor planar algebra $\SA=\{\SA_{n,\pm}\}_{n\in \mathbb{N}}$, such that $\SA_{n,+}=\hom(1,\gamma^n)$.
This is a part of Theorem \ref{Thm:MN Self-duality}. We show that the planar algebra is unshaded by constructing a planar algebraic *-isomorphism from $\SA_{n,-}$ to $\SA_{n,+}$ in \S \ref{Sec:M self-duality}.

The modular conjugation $\theta_1$ defines the involution $*$ of the subfactor planar algebra $\SA$.
In the planar algebra $\SA_{n,+}$, the element $\delta^{\frac{n}{2}} \mu_n$ is represented by
$$
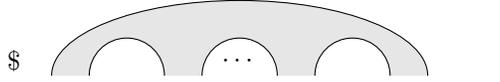
\begin{tikzpicture}
\fill[gray!20] (0,0)--++(.5,0) arc (180:0:.5)--++(.5,0) arc (180:0:.5)--++(.5,0) arc (180:0:.5)--(5,0) arc (0:180:2.5 and 1);
  \draw (.5,0) arc (180:0:.5);
  \draw (2,0) arc (180:0:.5);
  \draw (3.5,0) arc (180:0:.5);
  \draw (5,0) arc (0:180: 2.5 and 1);
\node at (2.5,0.2) {$\cdots$};
\node at (-.5,0.2) {$\$$};
\node at (6,0) {$,$};
\end{tikzpicture}
$$
where the diagram has $2n$ boundary points at the bottom.

\begin{remark}
Convention: We omit the $\$$ sign of the planar diagram if it is on the left.
\end{remark}

The action of any planar tangle on $\SA_{\cdot,+}$ is a composition of the following 6 elementary ones, for $n,\ell \geq0$:
\begin{itemize}

\item The rotation $\rho:\SA_{n,+} \to \SA_{n,+}$,
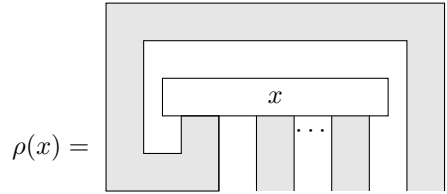
\begin{align*}
\rho(x)&=
\raisebox{-.5cm}{
\begin{tikzpicture}
\fill[gray!20] (1,0)--++(0,-.5)--++(-.5,0)--++(0,1.5)--++(3.5,0)--++(0,-2)--(4.5,-1)--++(0,2.5)--++(-4.5,0)--++(0,-2.5)--++(1.5,0)--++(0,1)--(1,0);
\draw (1,0)--++(0,-.5)--++(-.5,0)--++(0,1.5)--++(3.5,0)--++(0,-2);
\draw (4.5,-1)--++(0,2.5)--++(-4.5,0)--++(0,-2.5)--++(1.5,0)--++(0,1)--(1,0);
  \foreach \x in {2,3} {
  \fill[gray!20] (\x,0)--++(0,-1)--++(.5,0)--++(0,1)--++(-.5,0);
  \draw (\x,0)--++(0,-1);
  \draw (\x+.5,0)--++(0,-1);
  }
\draw (1+.5,0)--++(0,-1);
\draw (.75,0) rectangle (3.75,.5);
\node at (2.25,.25) {$x$};
\node at (2.75,-0.2) {$\cdots$};
\node at (.75,-0.2) {$~$};
\node at (.25,-0.7) {$~$};
\end{tikzpicture}
}
\end{align*}

\item The wedge product $\wedge:\SA_{n,+} \otimes \SA_{\ell,+} \to \SA_{n_\ell,+}$,
\begin{align*}
x\wedge y&=
\raisebox{-.5cm}{
\begin{tikzpicture}
  \foreach \x in {2,3} {
  \fill[gray!20] (\x,0)--++(0,-.5)--++(.5,0)--++(0,.5)--++(-.5,0);
  \draw (\x,0)--++(0,-.5);
  \draw (\x+.5,0)--++(0,-.5);
  }
\draw (1.75,0) rectangle (3.75,.5);
\node at (2.75,.25) {$x$};
\node at (2.75,-0.2) {$\cdots$};
\begin{scope}[shift={(2.25,0)}]
  \foreach \x in {2,3} {
  \fill[gray!20] (\x,0)--++(0,-.5)--++(.5,0)--++(0,.5)--++(-.5,0);
  \draw (\x,0)--++(0,-.5);
  \draw (\x+.5,0)--++(0,-.5);
  }
\draw (1.75,0) rectangle (3.75,.5);
\node at (2.75,.25) {$y$};
\node at (2.75,-0.2) {$\cdots$};
\end{scope}
\end{tikzpicture}
}
\end{align*}

\item
The inclusion $\iota_0:\SA_{n,\pm}\to \SA_{n+1,\pm}$,
\begin{align*}
\iota_0(x)&=\delta^{-1/2}
\raisebox{-.5cm}{
\begin{tikzpicture}
  \foreach \x in {2,3} {
  \fill[gray!20] (\x,0)--++(0,-.5)--++(.5,0)--++(0,.5)--++(-.5,0);
  \draw (\x,0)--++(0,-.5);
  \draw (\x+.5,0)--++(0,-.5);
  }
\fill[gray!20] (1,-.5) arc (180:0:.25);
\draw (1,-.5) arc (180:0:.25);
\draw (.75,0) rectangle (3.75,.5);
\node at (2.25,.25) {$x$};
\node at (2.75,-0.2) {$\cdots$};
\node at (.75,-0.2) {$~$};
\node at (.25,-0.2) {$~$};
\end{tikzpicture}.
}
\end{align*}

\item
The contraction $\phi_0:\SA_{n+1,\pm}\to \SA_{n,\pm}$,
\begin{align*}
\phi_0(x)&=\delta^{-1/2}
\raisebox{-.5cm}{
\begin{tikzpicture}
  \foreach \x in {2,3} {
  \fill[gray!20] (\x,0)--++(0,-.5)--++(.5,0)--++(0,.5)--++(-.5,0);
  \draw (\x,0)--++(0,-.5);
  \draw (\x+.5,0)--++(0,-.5);
  }
\fill[gray!20] (1,0) arc (-180:0:.25);
\draw (1,0) arc (-180:0:.25);
\draw (.75,0) rectangle (3.75,.5);
\node at (2.25,.25) {$x$};
\node at (2.75,-0.2) {$\cdots$};
\node at (.75,-0.2) {$~$};
\node at (.25,-0.2) {$~$};
\end{tikzpicture}.
}
\end{align*}

\item
The inclusion $\iota_1:\SA_{n,\pm}\to \SA_{n+1,\pm}$,
\begin{align*}
\iota_1(x)&=\delta^{-1/2}
\raisebox{-.5cm}{
\begin{tikzpicture}
  \foreach \x in {3} {
  \fill[gray!20] (\x,0)--++(0,-.5)--++(.5,0)--++(0,.5)--++(-.5,0);
  \draw (\x,0)--++(0,-.5);
  \draw (\x+.5,0)--++(0,-.5);
  }
\fill[gray!20] (1,0) rectangle (2.5,-.5);
  \foreach \x in {1,2.5} {
  \draw (\x,0)--++(0,-.5);
  }
\fill[white!20] (2,-.5) arc (0:180:.25);
\draw (2,-.5) arc (0:180:.25);
\draw (.75,0) rectangle (3.75,.5);
\node at (2.25,.25) {$x$};
\node at (2.75,-0.2) {$\cdots$};
\node at (.75,-0.2) {$~$};
\node at (.25,-0.2) {$~$};
\end{tikzpicture}.
}
\end{align*}

\item The contraction $\phi_1:\SA_{n+1,\pm}\to \SA_{n,\pm}$,
\begin{align*}
\phi_1(x)&=\delta^{-1/2}
\raisebox{-.5cm}{
\begin{tikzpicture}
  \foreach \x in {3} {
  \fill[gray!20] (\x,0)--++(0,-.5)--++(.5,0)--++(0,.5)--++(-.5,0);
  \draw (\x,0)--++(0,-.5);
  \draw (\x+.5,0)--++(0,-.5);
  }
\fill[gray!20] (1,0) rectangle (2.5,-.5);
  \foreach \x in {1,2.5} {
  \draw (\x,0)--++(0,-.5);
  }
\fill[white!20] (2,0) arc (0:-180:.25);
\draw (2,0) arc (0:-180:.25);
\draw (.75,0) rectangle (3.75,.5);
\node at (2.25,.25) {$x$};
\node at (2.75,-0.2) {$\cdots$};
\node at (.75,-0.2) {$~$};
\node at (.25,-0.2) {$~$};
\end{tikzpicture}.
}
\end{align*}

\end{itemize}
The first two are isometries.
The last four are partial isometries.
These actions except the rotation can be written as contractions:
\begin{align*}
\iota_0(x)&=\mu_1 \wedge x\\
\phi_0(x)&=\mu_1 \wedge_1 x\\
\iota_1(x)&=\delta \mu_3 \wedge_1 x\\
\phi_1(x)&=\delta \mu_3 \wedge_2 x.
\end{align*}
Moreover, $\phi_k$ is the adjoint operator of $\iota_k$:
\begin{proposition}\label{Prop:inc-con-adjoint}
For $x\in \SA_{n,\pm},y\in \SA_{n+2,\pm}$, $k=0,1$, we have
\begin{align*}
<\iota_k (x),y>=<x,\phi_k (y)>.
\end{align*}
\end{proposition}
As a subfactor planar algebra, we have the involution on $\SA_{n,+}$ defined by the reflection $\theta_1$ which is an anti-isometry.

We have also these actions on the configuration space in the $Y$-direction. In particular, the rotation $\rho_2$ and the reflection $\theta_2$ preserves the size $m$, and they are defined on $\SA_{n,+}=\hom(1,\gamma^n)$.

\begin{theorem} \label{Thm:Iso-rho-theta}
The action $\rho_2$ on $\SA$ is a planar algebra $*$-isomorphism.
The action $\theta_2$ on $\SA$ is an anti-linear planar algebra $*$-isomorphism.
\end{theorem}

\begin{proof}
By Propositions \ref{Prop:Rho Theta commute},
$\rho_2$ and $\theta_2$ commute with $\rho$ and $\theta_1$.
By Propositions \ref{Prop:Rho Theta preserve} and  \ref{Prop:Rho Theta Contraction}, we have that  $\rho_2$ and $\theta_2$ commute with $\wedge$, $\iota_k$ and $\phi_k$, for $k=1,2$. Therefore they are (anti-linear) planar algebraic *-isomorphisms.
\end{proof}

Similarly we have the 6+1 elementary actions on $\SA_{\cdot,-}$ by switching the shading.
The string Fourier transform (SFT) $\F:\mathscr{S}_{n,\pm}\to \SA_{n,\mp}$ is an isometry given by a clockwise one-string rotation.
Applying the SFT, we can represent the element in $\SA_{n,-}$ by $\F(x)$ for $x$ in $\SA_{n,+}$ and derive the six elementary actions on $\SA_{\cdot,-}$ by actions on $\SA_{\cdot,+}$.

For an element $x\in \SA_{n,+}$
\raisebox{-.5cm}{
\begin{tikzpicture}
  \foreach \x in {1,2,3} {
  \fill[gray!20] (\x,0)--++(0,-.5)--++(.5,0)--++(0,.5)--++(-.5,0);
  \draw (\x,0)--++(0,-.5);
  \draw (\x+.5,0)--++(0,-.5);
  }
\draw (.75,0) rectangle (3.75,.5);
\node at (2.25,.25) {$x$};
\node at (1.75,-0.2) {$\cdots$};
\node at (.5,0.25) {$~$};
\node at (4,0) {$,$};
\end{tikzpicture}}
its SFT $\F(x)\in \SA_{n,-}$ is given by
\begin{align}
\F(x)&=
\raisebox{-1cm}{
\begin{tikzpicture}
\fill[gray!20] (1,0)--++(0,-.5)--++(-.5,0)--++(0,1.5)--++(3.5,0)--++(0,-2)--(4.5,-1)--++(0,2.5)--++(-4.5,0)--++(0,-2.5)--++(1.5,0)--++(0,1)--(1,0);
  \foreach \x in {2,3} {
  \fill[gray!20] (\x,0)--++(0,-1)--++(.5,0)--++(0,1)--++(-.5,0);
  \draw (\x,0)--++(0,-1);
  \draw (\x+.5,0)--++(0,-1);
  }
\draw (1+.5,0)--++(0,-1);
\draw (1,0)--++(0,-.5)--++(-.5,0)--++(0,1.5)--++(3.5,0)--++(0,-2);
\draw (.75,0) rectangle (3.75,.5);
\node at (2.25,.25) {$x$};
\node at (2.75,-0.2) {$\cdots$};
\node at (.75,-0.2) {$~$};
\node at (.25,-0.7) {$~$};
\end{tikzpicture}}.
\end{align}
Then $\rho=\F^2$. Moreover, $\iota_k:=\F^{-k}\iota_0 \F^k$, $1\leq k \leq 2n$, is adding a cap before the $k^{\rm th}$ boundary points,
and $\phi_k:=\F^{-k}\phi_0 \F^k$, $1\leq k \leq 2n$, is a contraction between the $k+1^{\rm th}$ and $k+2^{\rm th}$ boundary points.

\begin{notation}
  By the spherical property, we define $\phi_1$ on $\SA_{1,\pm}$ by $\phi_0$.
\end{notation}

For $x\in \SA_{n,+}, y \in \SA_{n',+}$, we define $x\star y \in \SA_{n+n',+}$ as
\begin{align}\label{Equ:star}
x\star y&=
\raisebox{-2cm}{
\begin{tikzpicture}
\begin{scope}[xscale=-1]
\fill[gray!20]   (3,0)--++(0,-1)--++(1.5,0)--++(0,2)--++(0.5,0)--++(0,-2.5)--(5.5,-1.5)--++(0,2.5)--++(3.5,0)--++(0,-1.5)--++(-.5,0)--++(0,.5)-- (8,0)--++(0,-1)--++(1.5,0)--++(0,2.5)--++(-4,0)--++(0,.5)--++(4.5,0)--++(0,-3.5)--(10.5,-1.5)--++(0,4)--++(-5.5,0)--++(0,-1)--++(-1,0)--++(0,-2)--++(-0.5,0)--++(0,0.5)--(3,0);
  \draw (3,0)--++(0,-1)--++(1.5,0)--++(0,2)--++(0.5,0)--++(0,-2.5);
  \draw (5.5,-1.5)--++(0,2.5)--++(3.5,0)--++(0,-1.5)--++(-.5,0)--++(0,.5);
  \draw (8,0)--++(0,-1)--++(1.5,0)--++(0,2.5)--++(-4,0)--++(0,.5)--++(4.5,0)--++(0,-3.5);
  \draw (10.5,-1.5)--++(0,4)--++(-5.5,0)--++(0,-1)--++(-1,0)--++(0,-2)--++(-0.5,0)--++(0,0.5);
\node at (10.75,-1.2) {$~$};
  \foreach \x in {1,2} {
  \fill[gray!20] (\x,0) rectangle (\x+.5,-1.5);
  \draw (\x,0)--++(0,-1.5);
  \draw (\x+.5,0)--++(0,-1.5);
  }
\draw (.75,0) rectangle (3.75,.5);
\node at (2.25,.25) {$y$};
\node at (1.75,-0.2) {$\cdots$};
\node at (3.75,-0.2) {$~$};
\node at (4,0) {$$};
\begin{scope}[shift={(5,0)}]
  \foreach \x in {1,2} {
  \fill[gray!20] (\x,0)--++(0,-1.5)--++(.5,0)--++(0,1.5)--++(-.5,0);
  \draw (\x,0)--++(0,-1.5);
  \draw (\x+.5,0)--++(0,-1.5);
  }
\draw (.75,0) rectangle (3.75,.5);
\node at (2.25,.25) {$x$};
\node at (1.75,-0.2) {$\cdots$};
\node at (3.75,-0.2) {$~$};
\node at (4,0) {$$};
\end{scope}
\end{scope}
\end{tikzpicture}}
\end{align}
Then
\begin{align}
\label{Equ:Fourier rho}
\rho \F(x)&= \F \rho(x), \\
\label{Equ:Fourier tensor}
\F(x) \wedge \F(y)&=\F (x\star y), \\
\label{Equ:Fourier phi}
\phi_k \F(x)&= \F \phi_{k+1}(x), \\
\label{Equ:Fourier iota}
\iota_{k+1} \F(x)&= \F \iota_{k}(x), \\
\label{Equ:Fourier *}
\theta_1(\F(x))&=\F \rho^{-1} \theta_1(x).
\end{align}

Recall that $\SA_{n,+}=\hom(1,\gamma^n)$, thus $\star:\hom(1,\gamma^{n})\otimes\hom(1,\gamma^{n'}) \to \hom(1,\gamma^{n+n'})$ is also defined.
From the planar algebra $\SA$ to category $\C^m$, the shaded strip becomes a $\gamma$-string.
Then Equation \eqref{Equ:star} becomes
\begin{align*}
x\star y&=\delta^2
\raisebox{-.5cm}{
\begin{tikzpicture}
\node at (1.3,.25) {$x$};
\draw (1.5,.25)--(2,0)--++(0,-.5);
\node at (1.75,-0.2) {$\cdots$};
\draw (1.5,.25)--(1.5,0)--++(0,-.5);
\draw (1.5,.25)--(1,0)--++(0,-.25)--++(-.5,0)--++(0,.75);
\begin{scope}[shift={(2.5,0)}]
\node at (1.3,.25) {$y$};
\draw (1.5,.25)--(2,0)--++(0,-.5);
\node at (1.75,-0.2) {$\cdots$};
\draw (1.5,.25)--(1.5,0)--++(0,-.5);
\draw (1.5,.25)--(1,0)--++(0,-.25)--++(-.5,0)--++(0,.75);
\end{scope}
\draw (.5,.5)--++(2.5,0);
\draw (2.5,-.5)--++(0,1.5)--++(-2.5,0)--++(0,-1.5);
\node at (2.75,.75) {$\mu_4$};
\end{tikzpicture}}.
\end{align*}

\section{Modular self-duality}\label{Sec:M self-duality}

\subsection{The self-duality of Jones-Wassermann subfactors}
Suppose that $\SA$ is the subfactor planar algebras of the Jones-Wassermann subfactor for a unitary modular tensor category $\C$. In this section, we construct a planar algebraic *-isomorphism from $\SA_{n,-}$ to $\SA_{n,+}$.
Then the subfactor planar algebra $\SA$ is unshaded. Equivalently the Jones-Wassermann subfactor is self-dual.

Induced by $\Phi$, we define $LL$ on $\hom(1,\gamma^n)\otimes \hom(1,\gamma^n)$.
Moreover, we use the following notation to simplify the diagram in Equation \eqref{Equ:bilinear form}:


\begin{align} \nonumber
&LL(a_{\vec{j}}(b_{\vec{i}}), a'_{\vec{j}}( b'_{\vec{i}})) \\ \nonumber
=&LL(a_{\vec{j}}\otimes b_{\vec{i}}, a'_{\vec{j}}\otimes b'_{\vec{i}}) \\ \nonumber
=& \textcolor{black}{\delta^{1-n}} \sqrt{d(X_{\vec{i},\vec{j}})d(X'_{\vec{i},\vec{j}})}
\\ \label{Equ:bilinear form on SA}
&\raisebox{0cm}{
\begin{tikzpicture}
\myGT{3}{1};
            \foreach \x in {1,2,3,4} {
            \draw[blue](\x,-0.3,0)--(\x,0,0);
            \node at (\x,-.4,0) {$b_{\x}^*$};
                \node[blue] at (\x,-0.3,0) {$\bullet$};
                }
                \draw[blue] (1,0,0)--(4,0,0);
                \node at (2.5,.05,0) {$a_{\vec{j}}$};
                \begin{scope}[shift={(.5,0,.5)}]
            \foreach \x in {1,2,3,4} {
            \draw[red](\x,-0.3,0)--(\x,0,0);
            \node at (\x,-.4,0) {$(b'_{\x})^*$};
                \node[red] at (\x,-0.3,0) {$\bullet$};
                }
                \draw[red] (1,0,0)--(4,0,0);
                \node at (2.5,.07,0) {$a'_{\vec{j}}$};
                \end{scope}
\end{tikzpicture}
}
\end{align}

\begin{notation}
  We use $A_n$ to denote an ONB of $\hom_{\C^m}(1,\gamma^n)$. We use $B$ to denote an ONB of $\hom_{\C}(\gamma,1)$.
\end{notation}

\begin{definition}[string Fourier transform]\label{Def:SFT}
We represent elements in $\SA_{n,-}$ as $\F(x)$ for $x\in \SA_{n,+}=\hom(1,\gamma^n)$.  We define $\Psi:\SA_{n,-}\to\SA_{n,+}$, $n\geq0$,
\be \label{Equ:string Fourier transform}
\Psi(\F(x))=\sum_{x'\in B_{n}} LL(x, \theta_2(x'))x'.
\ee
\end{definition}
When $n=0$, $\Psi$ maps 1 to 1. When $n=1$, $\Psi$ maps the canonical inclusion from 1 to $\gamma$ in $\SA_{1,-}$ to the canonical inclusion in $\SA_{1,+}$.
Let us prove that $\Psi$ commutes with the 6+1 elementary actions, so $\Psi$ is a planar algebraic *-isomorphism from $\SA_{n,-}$ to $\SA_{n,+}$. Then $\SA$ becomes an unshaded planar algebra. Moreover,  the map $\Psi\F$ in Equation \eqref{Equ:string Fourier transform} defines the SFT on the unshaded planar algebra $\SA_{n}$.

When $m=n=2$, $\displaystyle \gamma=\bigoplus_{X\in Irr(\C)} X\otimes \overline{X}$.
The vectors $\{ v_X\}_{X\in Irr(\C)}$ form an ONB of $\hom_{\C^2}(1,\gamma^2)$, where $v_X$
is the canonical inclusion from 1 to $(X\otimes \overline{X}) \otimes \overline{(X\otimes \overline{X})}$ in $\C^2$. We call the ONB $\{ v_X\}_{X\in Irr(\C)}$ a {\it standard basis} of $\hom_{\C^2}(1,\gamma^2)$.

The vector $v_X$ is independent of the choice of the representative of $X\otimes \overline{X}$ in $\gamma$. For convenience, we take $b^*$ to be the canonical inclusion from 1 to $X\otimes \overline{X}$ to indicate the multiplicity of $X\otimes \overline{X}$ in $\gamma$,
then $\overline{(X\otimes \overline{X})(b)}=(\overline{X} \otimes X)(\theta_1(b))$.

The following result is a consequence of the modular self-duality and our definition of the SFT.
\begin{proposition}\label{Prop:SFT}
The SFT on the standard basis of $\hom_{\C^2}(1,\gamma^2)$ is the same as the modular $S$ matrix of the MTC $\C$: for any $X,X'\in Irr{C}$,
\begin{align}
\langle \Psi\F(v_X),v_{X'} \rangle=S_{XX'}.
\end{align}
\end{proposition}

\begin{proof}
By Definition \ref{Def:SFT}, the matrix units of $\Psi\F$ on the basis $\{ v_X\}_{X\in Irr(\C)}$ is 
\begin{align*}
\langle \Psi\F(v_X),v_{X'} \rangle&= LL(x, \theta_2(x'))\\
&=\mu^{-\frac{1}{2}}
\raisebox{-1.5cm}{
\begin{tikzpicture}
\begin{scope}[scale=1]
\myGT{3}{1};
            \foreach \x in {1} {
\begin{scope}[shift={(.5,0,.5)}]
            \draw[red](\x,-0.3,0)--(\x,-0.3,-1);
\end{scope}
            \draw[blue](\x,-0.3,0)--(\x,-0.3,-1);
               \foreach \z in {0,1} {
\begin{scope}[shift={(.5,0,.5)}]
                    \drawWL{red}{\x,-.15,-\z}{\x,-0.3,-\z};
\end{scope}
                    \drawWL{blue}{\x,0,-\z}{\x,-0.3,-\z};
                 }}
\foreach \x in {2} {
\begin{scope}[shift={(.5,0,.5)}]
            \draw[red](\x,-0.3,0)--(\x,-0.3,-1);
\end{scope}
            \draw[blue](\x,-0.3,0)--(\x,-0.3,-1);
               \foreach \z in {0,1} {
\begin{scope}[shift={(.5,0,.5)}]
                    \drawWL{red}{\x,-.15,-\z}{\x,-0.3,\z-1};
                    \draw[red] (2,-.15,-1)--(2,-.3,0);
            \draw[red](2,-0.3,0)--(2,-0.3,-1);
\end{scope}
                    \drawWL{blue}{\x,0,\z-1}{\x,-0.3,-\z};
                 }}        
\foreach \z in {0,1} {
\node[blue] at (2,-0.3,-\z) {$\bullet$};
}
\foreach \z in {0,1} {
\begin{scope}[shift={(.5,0,.5)}]
                \drawWL{red} {1,-.15,-\z}{2,0,-\z};
                \drawWL{red} {1,0,-\z}{2,-.15,-\z};
                \draw[red] (1,0,-\z)--(2,0,-\z);
\end{scope}
                \drawWL{blue} {1,0,-\z}{2,0,-\z};
}       
 \foreach \z in {0,1} {
    \foreach \x in {1,2 } {
\begin{scope}[shift={(.5,0,.5)}]
    \node[red] at (\x,0,-\z) {$\bullet$};
    \node[red] at (\x,-0.3,-\z) {$\bullet$};
\end{scope}
    \node[blue] at (\x,0,-\z) {$\bullet$};
    \node[blue] at (1,-0.3,-\z) {$\bullet$};
    }}
   \node at (1-.2,-0.15,-0) {$X$};
   \draw[blue] (1-.05,-0.2+.05,-0)--(1,-0.2,-0)--(1+.05,-0.2+.05,-0);
   \node at (1.5-.2,-0.2,0.5) {$\overline{X'}$};
   \draw[red] (1.5-.05,-0.2+.00,.5)--(1.5,-0.2-.05,.5)--(1.5+.05,-0.2+.00,.5);
    \end{scope}
\end{tikzpicture}}
\\
&=\mu^{-\frac{1}{2}}
\raisebox{-1cm}{
\begin{tikzpicture}
\begin{scope}[scale=1]
\draw[blue] (0,0) arc (180:0:1);
\draw[WL,white] (1,0) arc (180:0:1);
\draw[red] (1,0) arc (180:-180:1);
\draw[WL,white] (0,0) arc (-180:0:1);
\draw[blue] (0,0) arc (-180:0:1);
\draw[red] (1-.1,.2)--(1,0)--(1.1,.2);
\draw[blue] (2-.1,.2)--(2,0)--(2.1,.2);
\node at (2-.3,0) {$X$};
\node at (1-.3,0) {$X'$};
\end{scope}
\end{tikzpicture}} \\
&=S_{XX'}
\end{align*}

\end{proof}


\begin{proposition}\label{Prop:Iso-rho-*}
For $x\in\hom(1,\gamma^n)$, $n\geq1$,
\begin{align*}
\Psi (\F (\rho(x))&=\rho \Psi( \F(x)),\\
\Psi (\F \rho^{-1} \theta_1(x))&=\theta_1(\Psi( \F(x))).
\end{align*}
\end{proposition}

\begin{proof}
By Propositions \ref{Prop:Rho Theta commute} and \ref{Prop:LLrhotheta}, we have
\begin{align}
&\Psi (\F (\rho(x))\\
=& \sum_{x'\in A_n} LL(\rho(x), \theta_2(x'))x'  \\
=& \sum_{x'\in A_n} LL(x, \rho^{-1}\theta_2(x')))x' \\
=& \sum_{x'\in A_n} LL(x, \theta_2\rho^{-1}(x'))\rho\rho^{-1}(x') \\
=&\rho \Psi( \F(x)).
\end{align}

Similarly we have
$\Psi (\F \rho^{-1} \theta_1(x))=\theta_1(\Psi( \F(x))).$
\end{proof}

\begin{lemma}\label{Lem:star}
For $x,x'\in \hom_{\C^m}(1,\gamma^n)$, $y,y'\in \hom_{\C^m}(1,\gamma^{\ell})$, $n,\ell\geq1$,
$$LL(x\star y , x'\wedge y')=LL(x, x') LL(y, y').$$
\end{lemma}

\begin{proof}
Take $x=a_{\vec{j}}(b_{\vec{i}})$, $x=a'_{\vec{j}}(b'_{\vec{i}})$, $y=c_{\vec{j}}(d_{\vec{i}})$ and $y'=c'_{\vec{j}}(d'_{\vec{i}})$.
Note that the boundary of a $Y$-configuration $b$ is a $Z$-configuration, denoted by $\vec{X}(b)$. It represents a simple sub object of $\gamma$ in $\C^m$.
For $Y$-configurations $b,d \in B$, we define $A_{b,d}$ to be an ONB of $\hom_{\C^m} (1, X(b) \otimes X(\theta_1(b_1))\otimes X(d)\otimes X(\theta_1(d_1)))$, a sub space of $\hom_{\C^m}(1,\gamma^4)$. Then

\begin{align*}
&LL(\vec{x}\star\vec{y} , \vec{x'}\wedge\vec{y'})\\
=&\sum_{b,d \in B_1, \alpha \in A_{b,d}}  \delta^{1-n-\ell} \delta^{2} \sqrt{\frac{d(b) d(d) d(b_{\vec{j}}) d(b'_{\vec{j}}) d(d_{\vec{j}}) d(d'_{\vec{j}})}{d(b_1) d(d_1)}} \overline{L(\alpha)}\\
&
\raisebox{-.5cm}{
\begin{tikzpicture}
\myGT{3}{1};
\node at (1.5,-.6) {$b_{2}$};
\node at (2,-.6) {$b_{3}$};
\node at (0,-.6) {$b$};
\draw[blue] (1,0)--(2,0);
\node at (1.5,.05) {$a_{\vec{j}}$};
\draw[blue] (2,0)--++(0,-.5);
\draw[blue] (1.5,0)--++(0,-.5);
\foreach \x in {1.5,2}{
\node[blue] at (\x,-.5) {$\bullet$};}
\draw[blue] (1,0)--++(0,-.25)--++(-.5,0)--++(0,.5);
\begin{scope}[shift={(.5,0,.5)}]
\draw[red] (1,0)--(2,0);
\node at (1.5,.08) {$a'_{\vec{j}}$};
\draw[red] (2,0)--++(0,-.5);
\draw[red] (1.5,0)--++(0,-.5);
\draw[red] (1,0)--++(0,-.5);
\begin{scope}[shift={(.5,0)},xscale=.5]
\foreach \x in {1,2,3}{
\node at (\x,-.6) {$b'_{\x}$};
}
\end{scope}
\foreach \x in {1,1.5,2}{
\node[red] at (\x,-.5) {$\bullet$};}
\end{scope}
\begin{scope}[shift={(2.5,0)}]
\node at (1.5,-.6) {$d_{2}$};
\node at (2,-.6) {$d_{3}$};
\node at (0,-.6) {$d$};
\draw[blue] (1,0)--(2,0);
\node at (1.5,.05) {$c_{\vec{j}}$};
\draw[blue] (2,0)--++(0,-.5);
\draw[blue] (1.5,0)--++(0,-.5);
\draw[blue] (1,0)--++(0,-.25)--++(-.5,0)--++(0,.5);
\foreach \x in {1.5,2}{
\node[blue] at (\x,-.5) {$\bullet$};}
\begin{scope}[shift={(.5,0,.5)}]
\draw[red] (1,0)--(2,0);
\node at (1.5,.08) {$c'_{\vec{j}}$};
\draw[red] (2,0)--++(0,-.5);
\draw[red] (1.5,0)--++(0,-.5);
\draw[red] (1,0)--++(0,-.5);
\begin{scope}[shift={(.5,0)},xscale=.5]
\foreach \x in {1,2,3}{
\node at (\x,-.6) {$d'_{\x}$};
}
\end{scope}
\foreach \x in {1,1.5,2}{
\node[red] at (\x,-.5) {$\bullet$};}
\end{scope}
\end{scope}
\draw[blue] (.5,.25)--++(2.5,0);
\draw[blue] (2.5,-.5)--++(0,1)--++(-2.5,0)--++(0,-1);
\node at (2.6,.35) {$\alpha$};
\node[blue] at (0,-.5) {$\bullet$};
\node[blue] at (2.5,-.5) {$\bullet$};
\end{tikzpicture}} \\
=&\sum_{b,d \in B_1, \alpha \in A_{b,d}}  \delta^{3-n-\ell} \sqrt{\frac{d(b) d(d) d(b_{\vec{j}}) d(b'_{\vec{j}}) d(d_{\vec{j}}) d(d'_{\vec{j}})}{d(b_1) d(d_1)}} \overline{L(\alpha)}\\&
\raisebox{-.5cm}{
\begin{tikzpicture}
\myGT{3}{1};
\node at (.5,-.6) {$\theta(b_{1})$};
\node at (1,-.6) {$b_{1}$};
\node at (1.5,-.6) {$b_{2}$};
\node at (2,-.6) {$b_{3}$};
\node at (0,-.6) {$b$};
\draw[blue] (1,0)--(2,0);
\node at (1.5,.05) {$a_{\vec{j}}$};
\draw[blue] (2,0)--++(0,-.5);
\draw[blue] (1.5,0)--++(0,-.5);
\foreach \x in {1,1.5,2}{
\node[blue] at (\x,-.5) {$\bullet$};}
\draw[blue] (1,0)--++(0,-.5);
\draw[blue] (.5,-.5)--++(0,.75);
\node[blue] at (.5,-.5) {$\bullet$};
\begin{scope}[shift={(.5,0,.5)}]
\draw[red] (1,0)--(2,0);
\node at (1.5,.08) {$a'_{\vec{j}}$};
\draw[red] (2,0)--++(0,-.5);
\draw[red] (1.5,0)--++(0,-.5);
\draw[red] (1,0)--++(0,-.5);
\begin{scope}[shift={(.5,0)},xscale=.5]
\foreach \x in {1,2,3}{
\node at (\x,-.6) {$b'_{\x}$};
}
\end{scope}
\foreach \x in {1,1.5,2}{
\node[red] at (\x,-.5) {$\bullet$};}
\end{scope}
\begin{scope}[shift={(2.5,0)}]
\node at (.5,-.6) {$\theta(d_{1})$};
\node at (1,-.6) {$d_{1}$};
\node at (1.5,-.6) {$d_{2}$};
\node at (2,-.6) {$d_{3}$};
\node at (0,-.6) {$d$};
\draw[blue] (1,0)--(2,0);
\node at (1.5,.05) {$c_{\vec{j}}$};
\draw[blue] (2,0)--++(0,-.5);
\draw[blue] (1.5,0)--++(0,-.5);
\draw[blue] (1,0)--++(0,-.5);
\draw[blue] (.5,-.5)--++(0,.75);
\node[blue] at (.5,-.5) {$\bullet$};
\foreach \x in {1,1.5,2}{
\node[blue] at (\x,-.5) {$\bullet$};}
\begin{scope}[shift={(.5,0,.5)}]
\draw[red] (1,0)--(2,0);
\node at (1.5,.08) {$c'_{\vec{j}}$};
\draw[red] (2,0)--++(0,-.5);
\draw[red] (1.5,0)--++(0,-.5);
\draw[red] (1,0)--++(0,-.5);
\begin{scope}[shift={(.5,0)},xscale=.5]
\foreach \x in {1,2,3}{
\node at (\x,-.6) {$d'_{\x}$};
}
\end{scope}
\foreach \x in {1,1.5,2}{
\node[red] at (\x,-.5) {$\bullet$};}
\end{scope}
\end{scope}
\draw[blue] (.5,.25)--++(2.5,0);
\draw[blue] (2.5,-.5)--++(0,1)--++(-2.5,0)--++(0,-1);
\node at (2.6,.35) {$\alpha$};
\node[blue] at (0,-.5) {$\bullet$};
\node[blue] at (2.5,-.5) {$\bullet$};
\end{tikzpicture}} \text{by Lemma \ref{Lem:contraction rule}}\\
=&\sum_{b,d \in B_1, \alpha \in A_{b,d}} \delta^{4} \frac{1}{d(b_1) d(d_1)}  |L(\alpha)|^2 LL(x, x')LL(y, y')\\
=&\sum_{d \in B_1} \delta^{-2} d(d) LL(x, x')LL(y, y') \text{~\quad\quad\quad\quad by Lemma \ref{Lem:contraction rule} and Equation \eqref{Equ:resolution of identity}}\\
=&LL(x, x')LL(y, y') \text{\quad\quad\quad\quad\quad\quad\quad\quad\quad\quad by Proposition \ref{Prop:dim gamma}.}
\end{align*}
By the linearity, the equation holds for any $x$ and $x'$.
(Here we give the pictures for $n=\ell=3$. One can figure out the general case.)
\end{proof}

\begin{proposition}\label{Prop:Iso-tensor}
For $x\in \hom(1,\gamma^n)$, $y\in \hom(1,\gamma^{n'})$, $n,n'\geq1$,
\begin{align*}
\Psi (\F (x\star y))&=\Psi(\F(x)) \wedge \Psi(\F(y)).
\end{align*}
\end{proposition}

\begin{proof}
By Lemma \ref{Lem:star},
\begin{align*}
  &\Psi (\F (x\star y))\\
 =& \sum_{x', y' \in B} LL(\vec{x}\star\vec{y} , \vec{x'}\otimes\vec{y'}) x'\otimes y'\\
 =& \sum_{x', y' \in B} LL(\vec{x}, \vec{x'}) LL(\vec{y}, \vec{y'}) x'\otimes y' \\
 =& \Psi(\F(x)) \wedge \Psi(\F(y)) .
\end{align*}

\end{proof}

\begin{lemma}\label{Lem:LLphi1}
For $x\in\hom(1,\gamma^n)$ and $x'\in\hom(1,\gamma^{n-1})$, $n\geq1$,
  $$LL(\phi_1(x), x')= LL(x, \iota_0(x')).$$
\end{lemma}

\begin{proof}
When $n=1$, the statement is obvious.

When $n\geq2$, suppose $x=a_{\vec{j}}(b_{\vec{i}})$ and $x'=a'_{\vec{j}}(b'_{\vec{i}})$.
For $Y$-configurations $b \in B$, we define $A_{b}$ to be an ONB of $\hom_{\C^m} (1, X(b) \otimes X(\theta_1(b_2))\otimes X(\theta_1(b_1)))$, a sub space of $\hom_{\C^m}(1,\gamma^3)$.
 Then by Lemma \ref{Lem:contraction rule} and Equation \eqref{Equ:resolution of identity}, we have
\begin{align}
&LL(\phi_1(x), x')\\
=& \sum_{b \in B_1, \alpha\in A_{b}} \delta^{1-(n-1)} \delta \delta^{-2} d(b) \sqrt{d(b_{\vec{j}})d(b'_{\vec{j}})}\\
&\raisebox{0cm}
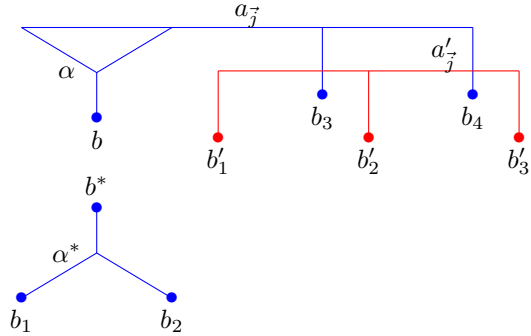
{
\begin{tikzpicture}
\myGT{3}{1};
            \foreach \x in {3,4} {
            \draw[blue](\x,-0.3,0)--(\x,0,0);
            \node at (\x,-.4,0) {$b_{\x}$};
                \node[blue] at (\x,-0.3,0) {$\bullet$};
                }
            \draw[blue](1,0,0)--++(0,-0,0)--++(.5,-0.2,0)--++(0,-0.2,0);
            \draw[blue](2,0,0)--++(0,-0,0)--++(-.5,-0.2,0);
                \node at (1.3,-.2,0) {$\alpha$};
                \node[blue] at (1.5,-.4,0) {$\bullet$};
                \node at (1.5,-.5,0) {$b$};
                \begin{scope}
                [shift={(0,-1.2,0)},yscale=-1]
            \draw[blue](1,0,0)--++(0,-0,0)--++(.5,-0.2,0)--++(0,-0.2,0);
            \draw[blue](2,0,0)--++(0,-0,0)--++(-.5,-0.2,0);
                \node at (1.3,-.2,0) {$\alpha^*$};
                \node[blue] at (1.5,-.4,0) {$\bullet$};
                \node at (1.5,-.5,0) {$b^*$};
                \node[blue] at (1,0,0) {$\bullet$};
                \node[blue] at (2,0,0) {$\bullet$};
                \node at (1,.1,0) {$b_1$};
                \node at (2,.1,0) {$b_2$};
                \end{scope}
                \draw[blue] (1,0,0)--(4,0,0);
                \node at (2.5,.05,0) {$a_{\vec{j}}$};
                \begin{scope}[shift={(1.5,0,.5)}]
            \foreach \x in {1,2,3} {
            \draw[red](\x,-0.3,0)--(\x,0,0);
            \node at (\x,-.4,0) {$b'_{\x}$};
                \node[red] at (\x,-0.3,0) {$\bullet$};
                }
                \draw[red] (1,0,0)--(3,0,0);
                \node at (2.5,.07,0) {$a'_{\vec{j}}$};
                \end{scope}
\end{tikzpicture}
}\\
=& \delta^{1-n} \sqrt{d(b_{\vec{j}})} \sqrt{d(b'_{\vec{j}})}\\
&\raisebox{0cm}
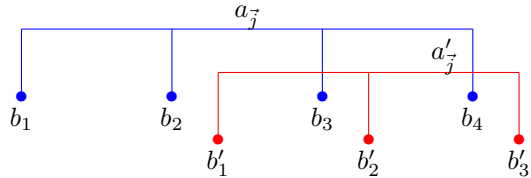
{
\begin{tikzpicture}
\myGT{3}{1};
            \foreach \x in {1,2,3,4} {
            \draw[blue](\x,-0.3,0)--(\x,0,0);
            \node at (\x,-.4,0) {$b_{\x}$};
                \node[blue] at (\x,-0.3,0) {$\bullet$};
                }
                \draw[blue] (1,0,0)--(4,0,0);
                \node at (2.5,.05,0) {$a_{\vec{j}}$};
                \begin{scope}[shift={(1.5,0,.5)}]
            \foreach \x in {1,2,3} {
            \draw[red](\x,-0.3,0)--(\x,0,0);
            \node at (\x,-.4,0) {$b'_{\x}$};
                \node[red] at (\x,-0.3,0) {$\bullet$};
                }
                \draw[red] (1,0,0)--(3,0,0);
                \node at (2.5,.07,0) {$a'_{\vec{j}}$};
                \end{scope}
\end{tikzpicture}} \\
&=  LL(x, \iota_0(x'))
\end{align}
The general case follows from the linearity.
\end{proof}

From the proof of Lemma \ref{Lem:LLphi1}, we see that the contraction $\phi_1$ on the configuration space is contracting the $Z$-configurations $X_{1,\vec{j}}$ and $X_{2,\vec{j}}$.
The diagrammatic representation of the contracted configuration is given by
\begin{align}\label{Equ:contraction}
&\delta^{1-n} \sqrt{d(X_{\vec{i},\vec{j}})}
\raisebox{-1cm}{
\begin{tikzpicture}
\myGT{3}{1};
            \foreach \x in {1,2,3,4} {
            \draw[blue](\x,-0.3,0)--(\x,0,0);
            \node at (\x,-.4,0) {$b_{\x}$};
                \node[blue] at (\x,-0.3,0) {$\bullet$};
                \draw[green] (\x-.1,-0.15,0)--++(.2,0,0);
                }
 \draw[green] (1-.1,-0.15,0)--++(1.2,0,0);
                \draw[blue] (1,0,0)--(4,0,0);
                \node at (2.5,.05,0) {$a_{\vec{j}}$};
\end{tikzpicture}
}.
\end{align}

\begin{lemma}\label{Lem:LLphi2}
For $x\in\hom(1,\gamma^n)$ and $x'\in\hom(1,\gamma^{n-1})$, $n\geq2$,
  $$LL(\phi_2(x), x')= LL(x, \iota_{1}(x')).$$
\end{lemma}

\begin{proof}
For $b'_-,b'_+ \in B$, take $A_{b'_-,b'_+}$ to be an ONB of $\hom(1, X(b'_-)\otimes X(b'_+) \otimes X(\theta_1(b'_1)))$. Then by Lemma \ref{Lem:contraction rule}, Equation \eqref{Equ:resolution of identity} and Proposition \ref{Prop:0 map'}, we have
\begin{align}
&LL(x, \iota_{1}(x'))\\
=&\sum_{b'_-,b'_+ \in B, \alpha\in A_{b'_-,b'_+}}  \delta^{1-n} \delta \delta^{-2}d(b'_-) d(b'_+) \sqrt{d(b_{\vec{j}})d(b'_{\vec{j}})}\\
&
\raisebox{0cm}{
\begin{tikzpicture}
\myGT{3}{1};
            \foreach \x in {1,2,3,4} {
            \draw[blue](\x,-0.4,0)--(\x,0,0);
            \node at (\x,-.5,0) {$b_{\x}$};
                \node[blue] at (\x,-0.4,0) {$\bullet$};
                }
                \draw[blue] (1,0,0)--(4,0,0);
                \node at (2.5,.05,0) {$a_{\vec{j}}$};
                \begin{scope}[shift={(.5,0,.5)}]
            \foreach \x in {3,4} {
                \draw[red](\x,-0.4,0)--(\x,0,0);
                \node[red] at (\x,-0.4,0) {$\bullet$};
                }
                    \node at (3,-.5,0) {$b'_{2}$};
                    \node at (4,-.5,0) {$b'_{3}$};
                    \draw[red] (2,0,0)--(4,0,0);
                \draw[red](1,-0.4,0)--(1,-.1,0);
                \node at (1,-.5,0) {$b'_{-}$};
                \node[red] at (1,-0.4,0) {$\bullet$};
                \draw[red](2,-0.4,0)--(2,0,0);
                \node at (2,-.5,0) {$b'_{+}$};
            \node[red] at (2,-0.4,0) {$\bullet$};
            \draw[red] (1,-.1,0)--(2,-.1,0);
            \node at (3,.1,0) {$a'_{\vec{j}}$};
                \node at (2-.1,0.0,0) {$\vec{\alpha}$};
                \end{scope}
            \begin{scope}[shift={(-2,-.4,0)},yscale=-1];
            \draw[red](1,-0.2,0)--(1,0,0);
                \node at (1,-.3,0) {$(b'_{-})^{*}$};
                \node[red] at (1,-0.2,0) {$\bullet$};
                \draw[red](2,-0.2,0)--(2,.2,0);
                \node at (2,-.3,0) {$(b'_{+})^{*}$};
            \node[red] at (2,-0.2,0) {$\bullet$};
            \draw[red] (1,0,0)--(2,0,0);
                \node at (2-.1,0.1,0) {$\vec{\alpha}$};
                \node[red] at (2,.2,0) {$\bullet$};
                \node at (2,.3,0) {$b_1$};
            \end{scope}
\end{tikzpicture}
}\\
=&\sum_{b'_- \in B} \delta^{-n}  d(b'_-)  \sqrt{d(b_{\vec{j}})d(b'_{\vec{j}})}\\
&\raisebox{0cm}{
\begin{tikzpicture}
\myGT{3}{1};
            \foreach \x in {1,2,3,4} {
            \draw[blue](\x,-0.4,0)--(\x,-.1,0);
            \node at (\x,-.5,0) {$b_{\x}$};
                \node[blue] at (\x,-0.4,0) {$\bullet$};
                }
                \draw[blue] (1,-.1,0)--(4,-.1,0);
                \node at (2.5,-.05,0) {$a_{\vec{j}}$};
                \begin{scope}[shift={(.5,0,.5)}]
            \foreach \x in {3,4} {
                \draw[red](\x,-0.4,0)--(\x,-.1,0);
                \node[red] at (\x,-0.4,0) {$\bullet$};
                }
                    \node at (3,-.5,0) {$b'_{2}$};
                    \node at (4,-.5,0) {$b'_{3}$};
                    \draw[red] (2,-.1,0)--(4,-.1,0);
                \draw[red](1,-0.4,0)--(1,-.2,0);
                \node at (1,-.5,0) {$b'_{-}$};
                \node[red] at (1,-0.4,0) {$\bullet$};
                \draw[red](2,-0.4,0)--(2,-.1,0);
                \node at (2,-.5,0) {$b'_{1}$};
            \node[red] at (2,-0.4,0) {$\bullet$};
            \draw[red] (1,-.2,0)--(1,-.1,0)--(1.75,-.1,0)--(1.75,-1,0)--(1,-1,0)--(1,-.8,0);
            \node[red] at (1,-.8,0) {$\bullet$};
            \node at (1,-.7,0) {$(b'_-)^{*}$};
            \node at (3,-.05,0) {$a'_{\vec{j}}$};
                \end{scope}
\end{tikzpicture}
}\\
&=\delta_{b_2,1_\gamma^*} \delta^2 \delta^{-n} \\
&\raisebox{0cm}{
\begin{tikzpicture}
\myGT{3}{1};
            \foreach \x in {1,3,4} {
            \draw[blue](\x,-0.4,0)--(\x,0,0);
            \node at (\x,-.5,0) {$b_{\x}$};
                \node[blue] at (\x,-0.4,0) {$\bullet$};
                }
                \draw[blue] (1,0,0)--(4,0,0);
                \node at (2.5,.05,0) {$a_{\vec{j}}$};
                \begin{scope}[shift={(.5,0,.5)}]
            \foreach \x in {2,3,4} {
                \draw[red](\x,-0.4,0)--(\x,0,0);
                \node[red] at (\x,-0.4,0) {$\bullet$};
                }
                    \node at (2,-.5,0) {$b'_{1}$};
                    \node at (3,-.5,0) {$b'_{2}$};
                    \node at (4,-.5,0) {$b'_{3}$};
                    \draw[red] (2,0,0)--(4,0,0);
            \node[red] at (2,-0.4,0) {$\bullet$};
            \node at (3,.1,0) {$a'_{\vec{j}}$};
                \end{scope}
\end{tikzpicture}
}\\
&= LL(\phi_2(x), x')
\end{align}
\end{proof}

\begin{lemma}\label{Lem:null}
  Suppose $H_1$ and $H_2$ are Hilbert spaces, and $T$ is an operator from $H_1$ to $H_2$. If $x \perp T(H_1)$ in $H_2$, then $T^*x=0$.
\end{lemma}

\begin{proof}
If $x \perp  T(H_1)$ in $H_2$, i.e., $<x,Ty>=0$, $\forall y \in H_2$, then $<T^*x,y>=0$. Thus $T^*x=0$.
\end{proof}

\begin{proposition}\label{Prop:Iso-con}
For  $0\leq k\leq 2n-2$, $x\in\hom(1,\gamma^n)$,
\begin{align*}
  \Psi (\F (\phi_{k+1}(x))&=\phi_k \Psi( \F(x)).
\end{align*}
\end{proposition}

\begin{proof}
For $k=0,1$,
\begin{align*}
&\Psi (\F (\phi_{k+1} (x))\\
=& \sum_{x'\in B_{n-1}} LL(\phi_{k+1}(x), \theta_2(x'))x'\\
=& \sum_{x'\in B_{n-1}} LL(x, \iota_k \theta_2(x')))x' && \text{by Lemmas \ref{Lem:LLphi1}, \ref{Lem:LLphi2}}\\
=& \sum_{x'\in B_{n-1}} LL(x, \theta_2\iota_k(x'))x' && \text{by Proposition \ref{Prop:Rho Theta commute}}\\
=&\sum_{x''\in \iota_k(B_{n-1})} LL(x, \theta_2 (x'')) \phi_k(x'') \\
=& \sum_{x''\in B_n} LL(x, \theta_2 (x'')) \phi_k(x'') && \text{by Proposition \ref{Prop:inc-con-adjoint} and Lemma \ref{Lem:null}}\\
=& \phi_k \Psi( \F(x)).
\end{align*}
The general case follows from Proposition \ref{Prop:Iso-rho-*}.
\end{proof}

\begin{proposition}\label{Prop:Iso-isometry}
The map $\Psi:\SA_{n,-}\to \SA_{n,+}$ is an isometry.
\end{proposition}

\begin{proof}
It is true for $n=0,1$ by definition.
When $n\geq2$,
  for $x$, $y$ in $\SA_{n,+}$,
\begin{align*}
&\langle \Psi\F(x), \Psi\F(y) \rangle \\
=& \delta^{n/2} \phi_0 \phi_1 \cdots \phi_{2n-1}  (\Psi\F(x)\wedge \Psi\F(y))\\
=& \delta^{n/2} \phi_0 \Psi\F (\phi_2 \cdots \phi_{2n}  (x\star y)) && \text{by Propositions \ref{Prop:Iso-tensor}, \ref{Prop:Iso-con}} \\
=& \delta^{n/2} \phi_0\phi_2 \cdots \phi_{2n}  (x\star y) \\
=& \langle x,y\rangle\\
=& \langle \F(x),\F(y)\rangle
\end{align*}.
\end{proof}

\begin{proposition}\label{Prop:Iso-inc}
For  $0\leq k\leq 2n-2$, $x\in\hom(1,\gamma^n)$,
\begin{align*}
\Psi (\F (\iota_{k}(x))&=\iota_{k+1} \Psi(\F(x)).
\end{align*}
\end{proposition}

\begin{proof}
By Propositions \ref{Prop:Iso-isometry}, \ref{Prop:inc-con-adjoint}, and \ref{Prop:Iso-con}, we have
\begin{align*}
&\langle \Psi\F \iota_k (x), y) \rangle.\\
=&\langle \iota_k (x), \Psi\F(y) \rangle\\
=&\langle x, \phi_k \Psi\F(y) \rangle\\
=&\langle x,\Psi \F (\phi_{k+1}(y)) \rangle\\
=&\langle \iota_{k+1} \Psi \F (x), y \rangle
\end{align*}
Therefore $\Psi (\F (\iota_{k}(x))=\iota_{k+1} \Psi(\F(x))$.

\end{proof}

\begin{theorem}\label{Thm:selfdual}
The map  $\Psi$ is a planar algebraic *-isomorphism from $\SA_{n,-}$ to $\SA_{n,+}$. Therefore, the $m$-interval Jones-Wassermann subfactor is self-dual for any $m\geq 1$.
\end{theorem}

\begin{proof}
We write an elements in $\SA_{\cdot,-}$ as  $x'=\F(x)$, $y'=\F(y)$, for $x, y\in\SA_{\cdot,+}$.

By Equation \eqref{Equ:Fourier rho} and Proposition \ref{Prop:Iso-rho-*},  $\Psi (\rho (x'))=\Psi (\F (\rho(x))=\rho \Psi( x')$.

By Equation \eqref{Equ:Fourier tensor} and Proposition \ref{Prop:Iso-tensor}, $\Psi( x'\wedge y')=\Psi(F(x\star y))=\Psi(x')\wedge \Psi(y')$.

By Equation \eqref{Equ:Fourier phi} and Proposition \ref{Prop:Iso-con},  $\Psi (\phi_{k} (x'))=\Psi (\F (\phi_{k+1}(x))=\phi_k \Psi( x')$.

By Equation \eqref{Equ:Fourier iota} and Proposition \ref{Prop:Iso-inc},  $\Psi (\iota_{k} (x') )=\Psi (\F (\iota_{k+1}(x))=\iota_k \Psi( x')$.

By Equation \eqref{Equ:Fourier *} and Proposition \ref{Prop:Iso-rho-*},  $\Psi(\theta_1 (x'))=\Psi (\F \rho^{-1} \theta_1(x))=\theta_1(\Psi( x'))$.

That means $\Psi$ commutes with the 6+1 elementary actions of planar algebras. So $\Psi$ is a planar algebraic *-isomorphism.
\end{proof}

\begin{remark}
The modularity is essential in the proof of the self-duality of Jones-Wassermann subfactors for the unitary MTC $\C$, so we call this property the modular self-duality of the MTC.
\end{remark}

\begin{remark}
Recall that $\rho_2$ is a planar algebraic *-isomorphism of $\SA_{\cdot,+}$ with periodicity $m$, then for each $k\in\Z_m$, $\Psi\rho_2^k$ is a planar algebraic *-isomorphism from $\SA_{\cdot,-}$ to $\SA_{\cdot,+}$. Therefore there are $k$ different ways to lift the shading of $\SA_{n,\pm}$. Each choice defines an unshaded subfactor planar algebra.
\end{remark}

\begin{remark}\label{Rem:orbifold}
From orbifold theory  it is easy to see that the Jones-Wassermann
subfactors for $n$ disjoint intervals are isomorphic to its dual as
subfactors. Here is a proof using orbifold theory: the dual of
$\pi_{1,\{0,1,...,n-1\}}$ is $\pi_{1,\{n-1,n-2,...,0\}},$ but
$\{n-1,n-2,...,0\}$  is conjugate to $\{0,1,...,n-1\}$ in $S_n$ via
$g(i)=n-i-1, i=0,1,...,n-1,$ hence $\pi_{1,\{n-1,n-2,...,0\}}\simeq
g \pi_{1,\{0,1,...,n-1\}} g^{-1}.$ We refer the readers to \cite{KacLonXu05} for details.
\end{remark}

If we take $\C$ to be the unitary modular tensor category, such that its fusion ring is the cyclic group $Z_d$ and its $S$ matrix is the discrete Fourier transform of $\Z_d$,
then two-box space of the two-interval Jones-Wassermann subfactor is isomorphic to $L^2(Z_d)$.
It is known that the usual multiplication and coproduct on the 2-box space in subfactor theory coincide with the multiplication and convolution on $L^2(Z_d)$.
In addition, we have shown that the SFT is the $S$ matrix which becomes the usual discrete Fourier transform. The modular self-duality reduces to the self-duality of $\Z_d$ on $L^2(\Z_d)$. Therefore the modular self-duality generalize and categorify the self-duality of finite abelian groups.

\subsection{Actions of planar tangles on the configuration space}
Motivated by the Jones-Wassermann subfactor, we obtain actions of planar tangles on the configuration spaces $\{Conf_{n,m}\}_{n,m\in\mathbb{N}}$ in both $X$- and $Y$-directions.
Moreover, these actions coincide with the geometric action on the lattices: The contraction tangle $\phi_1$ corresponds contractions of lattices to as shown in Equation \eqref{Equ:contraction}. The correspondence for the other 6+1 elementary tangles are more straightforward. Thus the actions of planar tangles in two different directions commute.
We call the (Hilbert) space $\{Conf_{n,m}\}_{n,m\in\mathbb{N}}$ equipped with such commutative actions of bidirectional planar tangles a \emph{bi-planar algebra} which we will study in the future.

Note that
\be
\D_+(x)=\sum_{x'\in B} \overline{LL(x, x')}x'=\theta_2\Psi(\F(x)).
\ee
Since $\theta_2$ is anti-isometry, we obtain Theorem \ref{Thm:MN Self-duality} from Proposition \ref{Prop:Iso-isometry}.

Moreover, $\theta_2$ commute with the action of planar tangles,
we have the following result corresponding to Propositions \ref{Prop:Iso-rho-*}, \ref{Prop:Iso-tensor}, \ref{Prop:Iso-con}, \ref{Prop:Iso-inc}:

\begin{proposition}\label{Prop:Duality map 6+1}
For $x\in Conf(\C)_{n,m}$, $y\in Conf(\C)_{\ell,m}$,
\begin{align*}
\D_+\rho(x)&=\rho \D_+(x),\\
\D_+\rho^{-1} \theta_1(x)&=\theta_1 \D_+(x),\\
\D_+ (x\star y)&=\D_+(x) \wedge \D_+(y),\\
\D_+ \phi_{k+1}(x)&=\phi_k \D_+(x),\\
\D_+\iota_{k}(x)&=\iota_{k+1} \D_+(x).
\end{align*}
\end{proposition}


\end{document}